\def\DateTime{September 22th, 2022}
\def\Version{Version $1.0$}

\documentclass[a4paper,reqno,10pt,english]{amsart}

\usepackage[T2A,T1]{fontenc}
\usepackage[english]{babel}
\usepackage[utf8]{inputenc}
\usepackage{latexsym,amscd,color}
\usepackage{amsmath,amsfonts,amssymb,mathrsfs,mathtools}
\usepackage{qsymbols,euscript,enumitem}
\usepackage{multirow,bigdelim}
\usepackage{bm,comment}
\usepackage{url,xspace}
\usepackage{verbatim}
\usepackage{tikz}
\usepackage{tikz-cd}
\usetikzlibrary{positioning,arrows,shapes}
\usepackage[normalem]{ulem}
\usepackage{manfnt} 
\usepackage{etoolbox}

\usepackage{hyperref}
\usepackage{xstring}

\makeindex
\numberwithin{equation}{section}

\theoremstyle{plain}
\author{Huayi CHEN}
\address{Universit\'e de Paris and Sorbonne Universit\'e, CNRS, INRIA,  IMJ-PRG, F-75013 Paris, France}
\email{huayi.chen@imj-prg.fr}
\author{Atsushi MORIWAKI}
\address{Department of Mathematics, Faculty of Science, Kyoto University, Kyoto, 606-8502, Japan}
\email{moriwaki@math.kyoto-u.ac.jp}
\title{Equidistribution theorem over an adelic curve}
\date{\DateTime,\ \rom{(}\Version\rom{)}}
\input xy
\xyoption{all}

\newcommand{\emptyinnprod}{\langle\kern.15em,\kern-.02em\rangle}

\def\sbullet{{\scriptscriptstyle\bullet}}

\def\rom{\textup}
\newcommand{\ZZ}{{\mathbb{Z}}}
\newcommand{\QQ}{{\mathbb{Q}}}
\newcommand{\RR}{{\mathbb{R}}}

\newcommand{\CC}{{\mathbb{C}}}
\newcommand{\PP}{{\mathbb{P}}}

\newcommand{\OO}{{\mathcal{O}}}

\newcommand{\Spec}{\operatorname{Spec}}

\newcommand{\ndot}{\raisebox{.4ex}{.}}

\newcommand{\rest}[2]{\left.{#1}\right\vert_{{#2}}}

\newcommand{\indic}{1\hspace{-0.25em}\mathrm{l}}
\newcommand\redsout{\bgroup\markoverwith{\textcolor{mred}{\rule[0.5ex]{2pt}{0.7pt}}}\ULon}
\def\colorsout#1{\bgroup\markoverwith{\textcolor{#1}{\rule[0.5ex]{2pt}{0.7pt}}}\ULon} 
\def\coloruline#1{\bgroup\markoverwith{\textcolor{#1}{\rule[-0.5ex]{2pt}{0.7pt}}}\ULon} 

\DeclareSymbolFont{bbold}{U}{bbold}{m}{n}
\DeclareMathSymbol{\bbalpha}{\mathord}{bbold}{"0B}
\DeclareMathSymbol{\bbbeta}{\mathord}{bbold}{"0C}
\DeclareMathSymbol{\bbgamma}{\mathord}{bbold}{"0D}
\DeclareMathSymbol{\bbdelta}{\mathord}{bbold}{"0E}
\DeclareMathSymbol{\bbespilon}{\mathord}{bbold}{"0F}
\DeclareMathSymbol{\bbzeta}{\mathord}{bbold}{"10}
\DeclareMathSymbol{\bbeta}{\mathord}{bbold}{"11}
\DeclareMathSymbol{\bbtheta}{\mathord}{bbold}{"12}
\DeclareMathSymbol{\bbiota}{\mathord}{bbold}{"13}
\DeclareMathSymbol{\bbkappa}{\mathord}{bbold}{"14}
\DeclareMathSymbol{\bblambda}{\mathord}{bbold}{"15}
\DeclareMathSymbol{\bbmu}{\mathord}{bbold}{"16}
\DeclareMathSymbol{\bbnu}{\mathord}{bbold}{"17}
\DeclareMathSymbol{\bbxi}{\mathord}{bbold}{"18}
\DeclareMathSymbol{\bbpi}{\mathord}{bbold}{"19}
\DeclareMathSymbol{\bbrho}{\mathord}{bbold}{"1A}
\DeclareMathSymbol{\bbsigma}{\mathord}{bbold}{"1B}
\DeclareMathSymbol{\bbtau}{\mathord}{bbold}{"1C}
\DeclareMathSymbol{\bbupsilon}{\mathord}{bbold}{"1D}
\DeclareMathSymbol{\bbphi}{\mathord}{bbold}{"1E}
\DeclareMathSymbol{\bbchi}{\mathord}{bbold}{"1F}
\DeclareMathSymbol{\bbpsi}{\mathord}{bbold}{"20}

\definecolor{ruby}{rgb}{0.88, 0.07, 0.37}
\definecolor{coolblack}{rgb}{0.0, 0.18, 0.39}
\definecolor{darkspringgreen}{rgb}{0.09, 0.45, 0.27}
\definecolor{emerald}{rgb}{0.31, 0.78, 0.47}
\definecolor{lavenderindigo}{rgb}{0.58, 0.34, 0.92}
\definecolor{mred}{rgb}{0.83, 0.0, 0.0}
\definecolor{indigo(web)}{rgb}{0.29, 0.0, 0.51}

\theoremstyle{plain}
\newtheorem{theo}{Theorem}[section]
\newtheorem{prop}[theo]{Proposition}
\newtheorem{coro}[theo]{Corollary}
\newtheorem{lemm}[theo]{Lemma}
\newtheorem{clai}[theo]{Claim}

\theoremstyle{definition}
\newtheorem{defi}[theo]{Definition}
\newtheorem{rema}[theo]{Remark}
\newtheorem{exem}[theo]{Example}

\begin{document}
\maketitle

\tableofcontents
\begin{abstract}
In this article, we introduce the notion of global adelic space of an arithmetic variety over an adelic curve and prove an equidistribution theorem for a generic sequence of subvarieties. As an application, we prove a Bogomolov type theorem for Abelian varieties over an adelic curve of characteristic $0$.
\end{abstract}

\section{Introduction}

In Arakelov geometry, equidistribution of small algebraic points in an arithmetic projective variety has firstly been studied in the work \cite{MR1427622} of Szipro, Ullmo and Zhang (see also the Bourbaki seminar review \cite{MR1627110} of Abbes), which has a fundamental importance in the resolution of Bogomolov's conjecture \cite{MR1609514,MR1609518} by Arakelov geometry method (see \cite{MR1478502} for another approach to the conjecture using Diophantine geometry). Let us remind the statement of the arithmetic equidistribution theorem in its classic form. Let $A$ be an abelian variety over a number field, $\overline L$ be a symmetric ample line bundle $L$ equipped with a positive adelic metric $\varphi$ such that the Arakelov height function with respect to $\overline L$ coincides with the N\'eron-Tate height. Let $(x_n)_{n\in\mathbb N}$ be a sequence of algebraic points of $A$ such that the N\'eron-Tate height of $x_n$ converges to $0$ (we say that such a sequence is \emph{small}). Then the Zariski closure $X$ of $(x_n)_{n\in\mathbb N}$ is the translation of an abelian subvariety of $A$ by a torsion point. Moreover, if in addition any subsequence of $(x_n)_{n\in\mathbb N}$ is Zariski closed in $X$, then, for any Archimedean place $\sigma$ of the number field, the Borel measure $\delta_{x_n,\sigma}$ on $X_\sigma(\mathbb C)$ of taking the average on the Galois orbit of $x_n$ converges weakly to the Monge-Amp\`ere measure $c_1(L_\sigma,\varphi_\sigma)^{\dim(X)}$ on $X_\sigma(\mathbb C)$. This equidistribution theorem has then been generalized in various contexts. We refer the readers to \cite{MR1779799} for the case where the base field is a finitely generated extension of $\mathbb Q$, to \cite{MR1832991,MR4404794} for the case of a semi-abelian variety, to \cite{MR2114791,MR2245872} for equidistribution of a small sequence of sub-varieties, to \cite{MR2164622,MR2221116,MR2244226} for the case of a dynamical system on a projective line, and to \cite{MR2244803} for an equidistribution theorem of a small sequence of algebraic points in the analytic variety over a non-Archimedean place.  We also refer to
\cite{MR2457191,MR2506588} for similar results over function fields. In \cite{MR2425137}, an arithmetic analogue of Siu's inequality has been proved, which leads to an equidistribution theorem with a weaker condition on the metrics of the adelic line bundle.  

In this article, we revisit the equidistribution of a small sequence of subvarieties in the setting of Arakelov geometry over an adelic curve. Let $K$ be a countable perfect field and $\phi=(|\ndot|_\omega)_{\omega\in\Omega}$ be a family of absolute values of $K$ which is parametrized by a measure space $(\Omega,\mathcal A,\nu)$, such that, for any $a\in K^{\times}$, the function $(\omega\in\Omega)\mapsto \ln|a|_\omega$ is $\nu$-integrable and of integral $0$. We call the data $S=(K,(\Omega,\mathcal A,\nu),\phi)$ a \emph{proper adelic curve}. Let $X$ be an integral projective scheme over $\Spec K$ and $d$ be the dimension of $X$. Let $\overline L=(L,\varphi)$ be an adelic line bundle on $X$, namely an invertible $\mathcal O_X$-module $L$ together with a family $\varphi=(\varphi_\omega)_{\omega\in\Omega}$ of metrics on $L_\omega$ satisfying dominancy and measurability conditions. We assume in addition that $L$ is semi-ample (namely a tensor power of $L$ is generated by global sections), $\deg_L(X)=(L^d)>0$ and $\varphi$ is semi-positive. The data $\overline L$ permit to construct an arithmetic intersection number $(\overline L|_Y)_S^{\dim(Y)+1}$ for any integral closed subscheme $Y$ of $X$, which can be written as an integral over $\Omega$ of local intersection numbers. In the case where $\deg_L(Y)=(\overline L|_Y)^{\dim(Y)}>0$, the \emph{normalized height} of $Y$ with respect to $\overline L$ is defined as  
\[h_{\overline L}(Y)=\frac{(\overline L|_Y)_S^{\dim(Y)+1}}{(\dim(Y)+1)\deg_L(Y)}.\] 
Let $Y$ be an integral closed subscheme of $X$ such that $\deg_L(Y)>0$. For any $\omega\in\Omega$, we denote by $\delta_{\overline L,Y,\omega}$ the Radon measure on $X$ such that, for any continuous function $f$ on the analytic space $X_\omega^{\mathrm{an}}$,
\[\int_{X_\omega^{\mathrm{an}}}f(x)\,\delta_{\overline L,Y,\omega}(\mathrm{d}x)=\frac{1}{\deg_L(Y)}\int_{Y_\omega^{\mathrm{an}}}f(y)\,c_1(L_\omega|_{Y_\omega},\varphi_{\omega}|_{Y_\omega})^{\dim(Y)}(\mathrm{d}y).\]
In the case where $|\ndot|_{\omega}$ is non-Archimedean,  the Monge-Amp\`ere measure \[c_1(L_\omega|_{Y_\omega},\varphi_{\omega}|_{Y_\omega})^{\dim(Y)}(\mathrm{d}y)\] has been constructed in \cite[Definition 2.4]{MR2244803}.

Note that, if one modifies the metrics $\varphi_\omega$ for $\omega$ belonging to a set of measure $0$, the height of subvarieties of $X$ does not change. However the local Monge-Amp\`ere measure can be modified by this procedure. Hence it is not adequate to consider a local equidistribution problem with respect to a single place $\omega$ unless the set $\{\omega\}$ belongs to $\mathcal A$ and has a positive measure with respect to $\nu$. We therefore introduce the following global version of Monge-Amp\`ere measure. Let $\Omega'$ be an element of $\mathcal A$ such that $\nu(\Omega')>0$. We denote by $X_{\Omega'}^{\mathrm{an}}$ the disjoint union $\coprod_{\omega\in\Omega'}X_\omega^{\mathrm{an}}$ of local analytifications indexed by $\Omega'$. We equipped this set with a suitable $\sigma$-algebra $\mathcal B_{X,\Omega'}$ so that the canonical projection map $X_{\Omega'}^{\mathrm{an}}\rightarrow \Omega'$ sending the elements of $X_\omega^{\mathrm{an}}$ to $\omega$ gives a fibration of measurable spaces. It turns out that local Monge-Amp\`ere measures mentioned above form a disintegration of a measure on $(X_{\Omega'}^{\mathrm{an}},\mathcal B_{X,\Omega'})$ over $\nu|_{\Omega'}$: for any integral closed subscheme $Y$ of $X$ such that $\deg_L(Y)>0$, we denote by $\delta_{\overline L,Y,\Omega'}$ the measure on $(X_{\Omega'}^{\mathrm{an}},\mathcal B_{X,\Omega'})$ which is defined as 
\[\int_{X_{\Omega'}^{\mathrm{an}}}f(x)\,\delta_{\overline L,Y,\Omega'}(\mathrm{d}x):=\int_{\Omega'}\bigg(\int_{X_\omega^{\mathrm{an}}}f(x)\,\delta_{\overline L,Y,\omega}(\mathrm{d}x)\bigg)\,\nu(\mathrm{d}\omega).\] 
From a functional point of view, one can consider $\delta_{\overline L,Y,\Omega'}$ as a linear form on the vector space of adelic families
of continuous functions on $X$.  
Denote by $\mathscr C^0_{\mathrm{a}}(X)$ the set of families $f = (f_{\omega})_{\omega \in \Omega}$ of 
continuous functions on $X$ such that $(\mathcal O_X, (\mathrm{e}^{-f_\omega}|\ndot|_{\omega})_{\omega \in \Omega})$ 
forms an adelic line bundle on $X$. Note that $f$ yields a measurable function $f_{\Omega}$ on $X_{\Omega}^{\mathrm{an}}$ 
given by $f_{\Omega}(x) = f_{\omega}(x)$ for $x \in X_\omega^{\mathrm{an}}$. 
We denote by $\mathscr C^0_{\mathrm{a}}(X;\Omega')$ the vector subspace of $\mathscr C^0_{\mathrm{a}}(X)$ consisting of $f\in \mathscr C^0_{\mathrm{a}}(X)$ such that $f_\omega=0$ for any $\omega\in\Omega\setminus\Omega'$. Then 
\[(f\in \mathscr C^0_{\mathrm{a}}(X;\Omega')) \longmapsto \int_{X_{\Omega'}^{\mathrm{an}}}f(x)\,\delta_{\overline L,Y,\Omega'}(\mathrm{d}x)\]
defines a linear functional on $\mathscr C^0_{\mathrm{a}}(X;\Omega')$. The main result of the article is the following (see Theorem \ref{theorem:equidistribution}).

\begin{theo}\label{Thm: equidistribution theorem}
Let $X$ be an integral projective scheme of dimension $d$ over $\Spec K$ and $\overline L=(L,\varphi)$ be an adelic line bundle on $X$ such that $L$ is semi-ample, $(L^d)>0$ and $\varphi$ is semi-positive. Let $(Y_n)_{n\in\mathbb N}$ be a sequence of integral closed subschemes of $X$, such that each of its subsequences is Zariski dense in $X$, and that $h_{\overline L}(Y_n)$ is well-defined and converges to $h_{\overline L}(X)$ when $n\rightarrow+\infty$. Then, for any $\Omega'\in\mathcal A$ such that $\nu(\Omega')>0$, the sequence of measures $(\delta_{\overline L,Y_n,\Omega'})_{n\in\mathbb N}$, viewed as a sequence of linear functionals on $\mathscr C^0_{\mathrm{a}}(X;\Omega')$, converges pointwisely to $\delta_{\overline L,X,\Omega'}$.
\end{theo}

The proof of the theorem is inspired by the original work of Szipro, Ullmo and Zhang, the subvariety version of Autissier, together with the differentiability interpretation introduced in \cite{MR2835335}. The idea relies on the following simple observation. Let $V$ be a finite-dimensional real vector space, $x_0$ be an element of $V$, and $f$ and $g$ be two real-valued functions on $U$ such that $f(x)\geqslant g(x)$ for any $x\in V$. Assume $f$ is concave on $V$, $g$ is G\^ateaux differentiable at $x_0$, and $f(x_0)=g(x_0)$. Then the function $f$ is also G\^ateaux differentiable at $x_0$ and its differential identifies with that of $g$. Concretely in the case of the equidistribution problem, we consider, for any integral closed subscheme $Y$ of $X$ such that $\deg_L(Y)>0$, the linear functional $\Phi_{Y}:\mathscr C^0_{\mathrm{a}}(X;\Omega')\rightarrow\mathbb R$ which sends $f\in\mathscr C^0_{\mathrm{a}}(X;\Omega')$ to
\[\frac{\operatorname{\widehat{vol}}_{\chi}((L,\varphi+f)|_Y)}{(\dim(Y)+1)\deg_L(Y)},\]
where $\operatorname{\widehat{vol}}_{\chi}((L,\varphi+f)|_Y)$ denotes the $\chi$-volume of $(L,\varphi+f)|_Y$, which is defined as 
\[\lim_{n\rightarrow+\infty}\frac{\widehat{\deg}\Big(H^0(Y,L|_Y^{\otimes n}),\big(\|\ndot\|_{(n\varphi_{\omega}+nf_{\omega})|_{Y_\omega}}\big)_{\omega\in\Omega}\Big)}{n^{\dim(Y)+1}/(\dim(Y)+1)!}.\]
By the arithmetic Hilbert-Samuel formula, the value of $\Phi_{Y}$ at $0$ identifies with $h_{\overline L}(Y)$. Moreover, this functional is concave. Consider now a generic sequence $(Y_n)_{n\in\mathbb N}$ of integral closed subschemes of $X$ as in Theorem \ref{Thm: equidistribution theorem}. For any $f\in\mathscr C^0_{\mathrm{a}}(X;\Omega')$, let 
\[\Phi_{Y_\sbullet}(f):=\liminf_{n\rightarrow+\infty}\Phi_{Y_n}(f).\]
Since the functionals $\Phi_{Y_n}$ are concave, so is $\Phi_{Y_\sbullet}$. The sequence $(Y_n)_{n\in\mathbb N}$ being generic, the functional $\Phi$ is bounded from below by $\Phi_X$. Moreover, the hypothesis that $(h_{\overline L}(Y_n))_{n}$ converges to $h_{\overline L}(X)$ shows that $\Phi_{Y_\sbullet}(0)=\Phi_X(0)$. Therefore, we deduce from the differentiability of $\Phi_X$ the equidistribution result. Note that the equality $\Phi_{Y_\sbullet}(0)=\Phi_X(0)$ is not always satisfied. In general, for any generic sequence $(Y_n)_{n\in\mathbb N}$, the limit inferior of $\Phi_{Y_n}(f)$ when $n\rightarrow+\infty$ is always bounded from below by the asymptotic maximal slope of $(L,\varphi+f)$, which is defined as 
\[\widehat{\mu}_{\max}^{\mathrm{asy}}(L,\varphi+f)=\lim_{n\rightarrow+\infty}\frac{\widehat{\mu}_{\max}\Big(H^0(Y,L|_Y^{\otimes n}),\big(\|\ndot\|_{(n\varphi_{\omega}+nf_{\omega})|_{Y_\omega}}\big)_{\omega\in\Omega}\Big)}{n}.\] Moreover, the lower bound $\widehat{\mu}_{\max}^{\mathrm{asy}}(L,\varphi)$ of $\Phi_{Y_\sbullet}(0)$ is attained by a certain generic sequence $(Y_n)_{n\in\mathbb N}$ (see \cite[\S5.2.9]{CMHS}). In particular, if the function 
\[(f\in\mathscr C^0_{\mathrm{a}}(X;\Omega'))\longmapsto \widehat{\mu}_{\max}^{\mathrm{asy}}(L,\varphi+f)\]
is G\^ateaux differentiable at $0$, then the following relation holds 
\[\lim_{n\rightarrow+\infty}\delta_{\overline L,Y_n,\Omega'}(f)=\frac{\mathrm{d}}{\mathrm{d}t}\Big|_{t=0}\widehat{\mu}_{\max}^{\mathrm{asy}}(L,\varphi+tf).\]

\medskip

As an application of the equidistribution theorem, we consider Bogomolov's conjecture over a countable field of characteristic zero.
We assume that $K$ is algebraically closed field of characteristic zero, $\nu(\Omega_{\infty}) > 0$, and $\nu(\mathcal A)\not\subseteq\{0,1\}$.

\begin{theo}
Let $A$ be an abelian variety over $K$, $L$ be an ample and symmetric line bundle on $A$, and $\varphi$ be a family of semipositive metrics
of $A$ such that $(A, \varphi)$ is nef and $\varphi_{\omega}$ is the canonical metric of $L_{\omega}$ for each $\omega\in \Omega$.
If the essential minimum of $\rest{(L, \varphi)}{X}$ is zero, then $X$ is a translation of an abelian subvariety of $A$ by a closed point of N\'{e}ron-Tate height $0$, which is a torsion point provided that any finitely generated subfield of $K$ has Northcott's property.
\end{theo}

This theorem is a generalization of \cite[Theorem~8.1]{MR1779799}. 

The rest  of the article is organized as follows. In the second section, we prove the G\^{a}teaux  differentiability with respect to modification of metrics of the $\chi$-volume function at any adelic line bundle $\overline L=(L,\varphi)$ such that $L$ is semi-ample and big and $\varphi$ is semi-positive. We then deduce the measurability of certain fiber integrals. This measurability result is important in the construction of global adelic space, that we present in the third section. We then prove, in the fourth section, an equidistribution theorem for a generic sequence of subvarieties in the setting of adelic curves. Finally, in the fifth section, we deduce Bogomolov's conjecture over a countable field of characteristic zero.

\section{Differentiability of $\chi$-volume}

In this section, we fix an adelic curve $S=(K,(\Omega,\mathcal A,\nu),\phi)$ such that the underlying field $K$ is countable and perfect.
Let $X$ be an integral projective scheme over $\Spec K$ and $d$ be the dimension of $X$. We show that the $\chi$-volume function $\operatorname{\widehat{vol}}_{\chi}(\ndot)$ is G\^ateaux differentiable along any direction defined by metric families on the open cone $\operatorname{\widehat{Pic}}_A(X)$ of adelic line bundles $\overline L$ on $X$ such that $L$ is semi-ample and big.

\subsection{Function associated with a metric family}

\begin{defi}
For any invertible $\mathcal O_X$-module $L$, we denote by $\mathscr M(L)$ the set of metric families $\varphi$ on $L$ such that $(L,\varphi)$ forms an adelic line bundle on $X$. 
If $L_1$ and $L_2$ are two invertible $\mathcal O_X$-modules, and $(\varphi_1,\varphi_2)\in\mathscr M(L_1)\times\mathscr M(L_2)$, we denote by $\varphi_1+\varphi_2$ the tensor product of the metric families $\varphi_1$ and $\varphi_2$, which is an element of $\mathscr M(L_1\otimes L_2)$.
\end{defi}

\begin{defi}
Let $U$ be a non-empty Zariski open set of $X$.
Let \[\mathscr{C}^0(U) := \prod\limits_{\omega \in \Omega} C^0(U_\omega^{\mathrm{an}})\]
and $f = (f_{\omega})_{\omega \in \Omega} \in \mathscr{C}^0(U)$.
We say that $f$ is \emph{measurable} if the following conditions are satisfied (see \cite[Definition 6.1.27]{CMArakelovAdelic}):
\begin{enumerate}[label=\rm(\alph*)]
\item For any closed point $P$ of $U$, the function from $\Omega_{K(P)}$ to $\mathbb R$ sending $v\in\Omega_{K(P)}$ to $f(x_{P,v})$ is $\mathcal A_{K(P)}$-measurable, where $x_{P,v}$ denotes the point of $X_{\omega}^{\mathrm{an}}$ represented by $(P_v,|\ndot|_v)$, $P_v$ being the point of $X_\omega(K(P)_v)$ extending $P$. 
\item We consider the trivial absolute value on $K$ and the corresponding Berkovich analytification $X^{\mathrm{an}}$. For any $x\in X^{\mathrm{an}}$, 
whose underlying scheme point is of dimension $1$, and such that the absolute value (in the structure of $x$) of the residue field has a rational exponent, the function $(\omega\in\Omega_0)\mapsto f_\omega(x)$ is $\mathcal A|_{\Omega_0}$-measurable, where $\Omega_0$ denotes the set of $\omega\in\Omega$ such that $|\ndot|_{\omega}$ is trivial.  
\end{enumerate}   
Under the assumption $U = X$,
the family $f$ is said to be \emph{dominanted} if there exists an integrable function $g$ on $(\Omega,\mathcal A)$ such that 
$\sup_{x\in X_\omega^{\mathrm{an}}}|f_\omega|(x)\leqslant g(\omega)$ for all $\omega\in\Omega$.
Here we set
\[
\begin{cases}
\mathscr{C}^0_{\mathrm{m}}(U) := \{ f  \in \mathscr{C}^0(U) \mid \text{$f$ is measurable} \}, \\[1ex]
\mathscr{C}^0_{\mathrm{d}}(X) :=  \{ f  \in \mathscr{C}^0(X) \mid \text{$f$ is dominated} \}, \\[1ex]
\mathscr{C}^0_{\mathrm{a}}(X) := \{ f  \in \mathscr{C}^0(X) \mid \text{$f$ is measurable and dominated} \}.
\end{cases}
\]
Sometimes $\mathscr{C}^0(U)$, $\mathscr{C}^0_{\mathrm{m}}(U)$, $\mathscr{C}^0_{\mathrm{d}}(X)$ and $\mathscr{C}^0_{\mathrm{a}}(X)$
are denoted by $\mathscr{C}^0(U;\Omega)$, $\mathscr{C}^0_{\mathrm{m}}(U;\Omega)$, $\mathscr{C}^0_{\mathrm{d}}(X;\Omega)$ and $\mathscr{C}^0_{\mathrm{a}}(X;\Omega)$, respectively,
to emphasize the parameter space $\Omega$.
\end{defi}

For $f \in \mathscr{C}^0(X)$, a metric family of $\mathcal O_{X}$ given by
$(e^{-f_\omega}|\ndot|_{\omega})_{\omega \in \Omega}$ is denoted by $\mathrm{e}^{-f}$.
Note that a map given by $f \mapsto (\mathcal{O}_X, \mathrm{e}^{-f})$ yields
a bijection between $\mathscr{C}^0(X)$ and the set of all metric families of $\mathcal{O}_X$.

\begin{prop}\label{prop:dominated:criterion}
For $f \in \mathscr{C}^0(X)$, we have the following equivalence:
{\allowdisplaybreaks\[
\begin{cases}
\text{$f \in \mathscr{C}^0_{\mathrm{m}}(X)$ $\Longleftrightarrow$ $(\mathcal{O}_X, \mathrm{e}^{-f})$ is measurable}, \\
\text{$f \in \mathscr{C}^0_{\mathrm{d}}(X)$ $\Longleftrightarrow$ $(\mathcal{O}_X, \mathrm{e}^{-f})$ is dominated},\\
\text{$f \in \mathscr{C}^0_{\mathrm{a}}(X)$ $\Longleftrightarrow$ $(\mathcal{O}_X, \mathrm{e}^{-f})$ is an adelic line bundle}.
\end{cases}
\]}%
\end{prop}
\begin{proof}
The first equivalence is obvious. The second is
a consequence of \cite[Proposition 6.1.12]{CMArakelovAdelic} and the fact that the zero metric on $\mathcal O_X$ is dominated.
The third follows from the first and second.
\end{proof}

By abuse of notation, we often identify $\mathscr{C}^0_{\mathrm{a}}(X)$ with $\mathscr M(\mathcal O_X)$.
Let $\Omega'$ be a measurable subset of $\Omega$ (i.e. $\Omega' \in \mathcal A$), and $S ' = (K, (\Omega', \mathcal A', \nu'), \phi')$ be the restriction of $S$ to $\Omega'$.
Note that $S'$ is also an adelic curve. 
We consider a natural correspondence 
$\mathscr{C}^0(X;\Omega') \to \mathscr{C}^0(X;\Omega)$ given by the following way:
if $f  = (f_\omega)_{\omega \in \Omega'}$, then 
$\tilde{f} = (\tilde{f}_\omega)_{\omega \in \Omega}$ is defined to be
\[
\tilde{f}_\omega := \begin{cases}
f_{\omega} & \text{if $\omega \in \Omega'$}, \\
0 & \text{otherwise}.
\end{cases}
\]
Then, as a  corollary of the above proposition, we have the following.

\begin{coro}\label{coro:extension:by:zero}
The above correspondence yields
$\mathscr{C}_{\mathrm{a}}^0(X;\Omega')\subseteq \mathscr{C}_{\mathrm{a}}^0(X;\Omega)$.
\end{coro}

\subsection{Measurability of partial derivatives}
We assume that $K$ is algebraically closed and  $\Omega = \Omega_\infty$. 
We fix a root $\sqrt{-1}$ of the equation $T^2 + 1 = 0$ in $K$, and a family $( \iota_{\omega} )_{\omega \in \Omega_{\infty}}$ of embeddings $K \to \CC$ which satisfies the following conditions
(c.f. \cite[Lemma~4.2.5]{CMIntersection}):
\begin{enumerate}[label=\rm(\roman*)]
\item for any $\omega\in\Omega_\infty$,
$\iota_{\omega}(\sqrt{-1}) = i$, where $i\in\mathbb C$ denotes the usual imaginary unit,
\item for any $\omega \in \Omega_{\infty}$, $|\ndot|_{\omega} = |\iota_{\omega}(\ndot)|$,
\item for any $a \in K$, the function $(\omega \in \Omega_{\infty}) \mapsto
\iota_{\omega}(a)$ is measurable. 
\end{enumerate}

Let $\pi : X \to \PP^d_K$ be a finite projection and $V$ be an affine Zariski open set of $\PP^d_K$ such that
if we set $U = \pi^{-1}(V)$, then $U$ is smooth over $K$ and $\pi : U \to V$ is \'{e}tale.
Let \[\mathscr{C}^\infty(U) := \prod\limits_{\omega \in \Omega} C^\infty(U_\omega^{\mathrm{an}})\quad\text{and}\quad
\mathscr{C}^\infty_{\mathrm{m}}(U) := \mathscr{C}^\infty(U) \cap \mathscr{C}^o_{\mathrm{m}}(U).
\]
Let $x$ be a closed point of $U$. For any $\omega\in\Omega$, we denote by $x_\omega$ the unique point of $X_\omega^{\mathrm{an}}$ whose underly scheme point of $X_\omega$ lies over $x$.
Let $(z_j)_{j=1}^d$ be a coordinate of $V$. Note $\pi^*(z_1), \ldots, \pi^*(z_d)$ yields a local \'{e}tale coordinate of $U$ around $x$.
By abuse of notation, $\pi^*(z_1), \ldots, \pi^*(z_d)$ are denoted by $z_1, \ldots, z_d$. 

\begin{prop}\label{prop:measurable:partial:derivative}
If $f = (f_\omega) \in \mathscr{C}^\infty_{\mathrm{m}}(U)$, then the function given by
\[(\omega\in\Omega)\longmapsto \frac{\partial^2(f_{\omega})}{\partial z_j\partial\overline{z}_\ell}(x_{\omega})\]
is $\mathcal A$-measurable.
\end{prop}

\begin{proof}
We may assume that $\pi(x) = (0, \ldots, 0)$.
As $V$ is a Zariski open set of $\mathbb A_K^d$, we can find a non-negative integer $e$ and a non-zero polynomial
\[
h = \sum_{\substack{i_1, \ldots, i_d \in \mathbb{Z}_{\geqslant 0},\\
i_1 + \cdots + i_d \leqslant e}} a_{i_1, \ldots, i_{d}} X_1^{i_1} \cdots X_{d}^{i_{d}}\quad(a_{i_1, \ldots, i_{d}}  \in K)
\]
such that $h(0, \ldots, 0) \not= 0$ and
\[V'  := \{ (x_1, \ldots, x_d) \in K^d \mid h(x_1, \ldots, x_n) \not= 0 \} \subseteq V.\]
If we set
\[
h_{\omega} =  \sum_{i_1, \ldots, i_d \in \mathbb{Z}_{\geqslant 0}} \iota_\omega(a_{i_1, \ldots, i_{d}}) X_1^{i_1} \cdots X_{d}^{i_{d}}
\]
for each $\omega \in \Omega$, then
${V_\omega'}^{\mathrm{an}} = \{ (\zeta_1, \ldots, \zeta_d) \in \mathbb{C}^d \mid h_\omega(\zeta_1, \ldots, \zeta_n) \not= 0 \}$.
For a polynomial 
\[
g = \sum_{\substack{i_1, \ldots, i_d \in \mathbb{Z}_{\geqslant 0}, \\i_1+\cdots + i_d \leqslant e}} b_{i_1, \ldots, i_{d}} X_1^{i_1} \cdots X_{d}^{i_{d}} \in \CC[X_1, \ldots, X_d],
\]
we define $\rho(g)$ to be
\[
\rho(g) := \inf \{ (|\zeta_1|^2 + \cdots + |\zeta_d|^2)^{1/2} \mid  g(\zeta_1, \ldots, \zeta_d) = 0 \}.\] Note that $\rho(g) = 0$ if and only if $g(0, \ldots, 0) = 0$, and
$\rho$ is continuous with respect to the coefficients $(b_{i_1, \ldots, i_{d}})$, so
if we set $r_\omega = \rho(h_\omega)$, then $r_\omega > 0$ and the function given by $(\omega \in \Omega) \mapsto r_{\omega}$ is $\mathcal{A}$-measurable. Moreover,
\[W_{\omega} = \{ (\zeta_1, \ldots, \zeta_d) \in \mathbb{C}^d \mid |\zeta_1|^2 + \cdots + |\zeta_d|^2 < r_\omega^2 \}
 \subseteq {V_\omega'}^{\mathrm{an}}.\]%
As $W_{\omega}$ is simply connected and $\pi_{\omega}^{-1}(W_\omega)$ is \'etale over $W_\omega$,
$\pi_{\omega}^{-1}(W_\omega)$ is a disjoint union of connected open sets.
Let $T_{\omega}$ be the connected component of $\pi_{\omega}^{-1}(W_\omega)$ such that $x_{\omega} \in T_{\omega}$.
Then $\pi_\omega : T_\omega \to W_\omega$ is an isomorphism.

Let $n$ be a positive integer and $A_n := \{\omega \in \Omega \mid r_{\omega} \geqslant 1/n \} \in \mathcal A$.
Let \[(p_1+\sqrt{-1}q_1,\ldots,p_d+\sqrt{-1}q_d)\in \QQ(\sqrt{-1})^d\] such that $(p_1^2 + q_1^2) + \cdots + (p_d^2 + q_d^2) \leqslant 1/n^2$. Then, for each $\omega \in A_n$,
we can find $y_{\omega} \in T_{\omega}$ such that $z_\omega(y_\omega) = (p_1+i q_1,\ldots,p_d+ iq_d)$.
Further, since $y_{\omega} \in T_{\omega}$ for all $\omega \in A_n$, there exists $y \in U$ such that $z(y) = (p_1+\sqrt{-1}q_1,\ldots,p_d+\sqrt{-1}q_d)$ and
$y$ is the image of $y_\omega$ by $X_\omega \to X$ for all $\omega \in A_n$.
Therefore, for any $\varepsilon\in\mathbb Q$ with $0 < \varepsilon < 1/(\sqrt{2}n)$, we
 can find $y_{1,\varepsilon}, y_{2,\varepsilon}, y_{3,\varepsilon} \in U'(K)$ such that
\[z(y_{1,\varepsilon}) = \varepsilon e_j, \ z(y_{2,\varepsilon}) = \sqrt{-1}\varepsilon e_\ell,\  z(y_{3,\varepsilon} ) = \varepsilon e_j+\sqrt{-1}\varepsilon e_{\ell} \]
and 
\[
\lim_{\varepsilon\to 0} (y_{1,\varepsilon})_\omega  = x_\omega,\quad
\lim_{\varepsilon\to 0} (y_{2,\varepsilon})_\omega  = x_\omega\quad\text{and}\quad
\lim_{\varepsilon\to 0} (y_{3,\varepsilon})_\omega  = x_\omega.
\] 
Thus one obtains
\[\frac{\partial^2(f_{\omega})}{\partial z_j\partial\overline{z}_\ell}(x_{\omega})
=\lim_{\varepsilon\rightarrow 0}\frac{1}{\varepsilon^2}\Big[f_{\omega}((y_{3, \varepsilon})_{\omega})-
f_{\omega}((y_{1,\varepsilon})_\omega)-f_{\omega}((y_{2, \varepsilon})_\omega)+f_{\omega}(x_\omega)\Big] \]
for $\omega \in A_{n}$. Note that, for any rational point $y$ of $U$, the function $(\omega\in\Omega)\mapsto f_{\omega}(y_\omega)$ is $\mathcal A$-measurable. Therefore, if we set
\[
b_n(\omega) = \begin{cases}
{\displaystyle \frac{\partial^2(f_{\omega})}{\partial z_j\partial\overline{z}_\ell}(x_{\omega})} & \text{$\omega \in A_n$},\\
0 & \text{otherwise},
\end{cases}
\]
then $b_n$ is $\mathcal A$-measurable on $\Omega$, so
the assertion is proved because \[\lim_{n\to\infty} b_n(\omega)= \frac{\partial^2(f_{\omega})}{\partial z_j\partial\overline{z}_\ell}(x_{\omega}).\]
\end{proof}

\subsection{Relative volume and $\chi$-volume}

If $L$ is an invertible $\mathcal O_X$-module, we denote by 
\[V_\sbullet(L)=\bigoplus_{n\in\mathbb N}H^0(X,L^{\otimes n})\]
the total graded linear series of $L$. Recall that the \emph{volume} of $L$ is defined as
\[\operatorname{vol}(L)=\limsup_{n\rightarrow+\infty}\frac{\dim_K(H^0(X,L^{\otimes n}))}{n^d/d!}.\]
The invertible $\mathcal O_X$-module $L$ is said to be \emph{big} if $\operatorname{vol}(L)>0$.

If $L$ is a big invertible $\mathcal O_X$-module and $V_\sbullet$ is a graded linear series of $L$ (namely a graded sub-$K$-algebra of $V_\sbullet(L)$), we denote by $\Delta(V_\sbullet)$ the Newton-Okounkov body of $V_\sbullet$. We refer to \cite{MR2950767,MR2571958} for the construction of Newton-Okounkov body under the assumption that the function field of $X$ admits a $\mathbb Z^d$-valuation of one-dimension leaves over $K$, see also \cite{MR3938628,MR4104565} for the arithmetic construction which applies to the general case. Recall that $\Delta(V_\sbullet)$ is a closed convex subset of $\mathbb R^d$, whose Lebesgue measure is equal to
\[\frac{1}{d!}\operatorname{vol}(V_\sbullet):=\limsup_{n\rightarrow+\infty}\frac{\dim_K(V_n)}{n^d}.\] 
In the case where $V_\sbullet$ 
is the total graded linear series of $L$, the Newton-Okounkov body is denoted by $\Delta(L)$. Recall that, if $L$ and $L'$ are big invertible $\mathcal O_X$-modules, and $V_\sbullet$, $V_\sbullet'$ and $W_\sbullet$ are respectively graded linear series of $L$, $L'$ and $L\otimes L'$, such that 
\[\forall\,n\in\mathbb N,\quad V_n\cdot V_n'\subseteq W_n,\]
then one has 
\[\Delta(V_\sbullet)+\Delta(V_\sbullet'):=\{x+y\,:\,(x,y)\in\Delta(V_\sbullet)+\Delta(V_\sbullet')\}\subseteq\Delta(W_\sbullet).\]

Let $(L,\varphi)$ be an adelic line bundle on $X$. For any $n\in\mathbb N$, we denote by $\xi_{n\varphi}$ the norm family $(\|\ndot\|_{n\varphi_\omega})_{\omega\in\Omega}$ on $V_n(L)=H^0(X,L^{\otimes n})$, so that $(V_n(L),\xi_{n\varphi})$ forms an adelic vector bundle on $S$. Let $\mathcal F$ be the Harder-Narasimhan $\mathbb R$-filtration of this adelic line bundle. Recall that 
\[\forall\,t\in\mathbb R,\quad \mathcal F^t(V_n(L),\xi_{n\varphi})=\sum_{\begin{subarray}{c}
\boldsymbol{0}\neq F\subseteq V_n(L)\\
\widehat{\mu}_{\min}(\overline F)\geqslant t
\end{subarray}}F,\]
where $F$ runs over the set of non-zero vector subspaces of $V_n(L)$ such that the minimal slope of $F$ equipped with restricted norm family of $\xi_{n\varphi}$ is not less than $t$. Note that this $\mathbb R$-filtration determines an ultrametric norm on $V_n(L)$ (where we consider the trivial absolute value on $K$), which we denote by $\|\ndot\|_{n\varphi}$. Denote by $\|\ndot\|_{n\varphi,\operatorname{sp}}$ the corresponding spectral norm, which is defined as 
\[\|s\|_{n\varphi,\operatorname{sp}}:=\lim_{N\rightarrow+\infty}\|s^{N}\|_{nN\varphi}^{\frac 1N}.\]
For any $t\in\mathbb R$, let 
\[V_n^{\varphi,t}(L):=\{s\in V_n(L)\,:\,\|s\|_{n\varphi,\operatorname{sp}}\leqslant \mathrm{e}^{-nt}\}.\]
Note that $V_\sbullet^{\varphi,t}(L)$ is a graded linear series of $L$.

\begin{defi}Let $(L,\varphi)$ be an adelic line bundle on $X$ such that $L$ is big.
We call \emph{concave transform} of the metric family $\varphi$ the  function $G_{\varphi}:\Delta(L)\rightarrow\mathbb R$ defined as follows:
\[G_{\varphi}(x)=\sup\{t\in\mathbb R\,:\,x\in\Delta(V_\sbullet^{\varphi,t}(L))\}.\]
This function is concave. Moreover, for any $x\in\Delta(L)$, one has 
\[\liminf_{n\rightarrow+\infty}\frac{\widehat{\mu}_{\min}(V_n(L),\xi_{n\varphi})}{n}\leqslant G_{\varphi}(x)\leqslant \limsup_{n\rightarrow+\infty}\frac{\widehat{\mu}_{\max}(V_n(L),\xi_{n\varphi})}{n}.\] For any positive integer $n$, one has $\Delta(L^{\otimes n})=n\Delta(L)$ and \begin{equation}\label{Equ: concave transform homogeneity}
G_{n\varphi}(nx)=n G_{\varphi}(x).\end{equation}
In the case where
\[\liminf_{n\rightarrow+\infty}\frac{\widehat{\mu}_{\min}(V_n(L),\xi_{n\varphi})}{n}>-\infty,\] the following equality holds:
\begin{equation}\label{Equ: integral G is equal to vol chi}\int_{\Delta(L)}G_{\varphi}(x)\,\mathrm{d}x=\frac{1}{(d+1)!}\widehat{\mathrm{vol}}_\chi(L,\varphi).\end{equation} 
\end{defi}

\begin{defi}
We say that an invertible $\mathcal O_X$-module $L$ is \emph{slope-bounded} if there exists a metric family $\varphi\in\mathscr M(L)$ such that 
\[\liminf_{n\rightarrow+\infty}\frac{\widehat{\mu}_{\min}(V_n(L),\xi_{n\varphi})}{n}>-\infty.\]
Note that, for any element $\psi\in\mathscr M(L)$, one has 
\[\forall\,n\in\mathbb N,\quad \big|\widehat{\mu}_{\min}(V_n(L),\xi_{n\varphi})-\widehat{\mu}_{\min}(V_n(L),\xi_{n\psi})\big|\leqslant n\int_{\Omega}\sup_{x\in X_\omega^{\mathrm{an}}}|\varphi_\omega-\psi_\omega|(x)\,\nu(\mathrm{d}\omega) \]
and therefore 
\[\liminf_{n\rightarrow+\infty}\frac{\widehat{\mu}_{\min}(V_n(L),\xi_{n\psi})}{n}>-\infty.\]
\end{defi}

\begin{exem}
Let $L$ be a big invertible $\mathcal O_X$-module. Assume that the graded linear series $V_\sbullet(L):=\bigoplus_{n\in\mathbb N}V_n(L)$ is of finite type over $K$. Then the invertible $\mathcal O_X$-module $L$ is slope-bounded. In particular, semiample and big invertible $\mathcal O_X$-modules are slope-bounded (cf. Remark~\ref{rema:semiample:of:finite:type} below). 
\end{exem}

\begin{rema}\label{rema:semiample:of:finite:type}
Let $L$ be a semiample invertible $\mathcal O_X$-module. Then $V_\sbullet(L)$ is of finite type over $K$.
Indeed, there exists a surjective morphism $f : X \to Y$ of projective integral schemes over $K$, an ample invertible $\OO_Y$-module $A$ and a positive integer $a$ such that $L^{\otimes a} = f^*(A)$.  Thus, by \cite[Lemma~3.4.1]{CMHS}, $R = V_\sbullet(L^{\otimes a})$ is of finite type over $K$ and
\[M_i = \bigoplus_{n=0}^{\infty} H^0(X, L^{\otimes i} \otimes L^{\otimes n a})\] is finitely generated over $R$ for every $0 \leqslant i < a$.
Therefore, $V_\sbullet(L) = M_0 \oplus \cdots \oplus M_{a-1}$ is also finitely generated over $R$, and hence the assertion follows.
\end{rema}

\begin{prop}\label{Pro: super additive G} 
Let $L_1$ and $L_2$ be  invertible $\mathcal O_X$-modules. If $(\varphi_1,\varphi_2)$ is an element of $\mathscr M(L_1)\times\mathscr M(L_2)$, then the following inequality holds:
\[\forall\,(x,y)\in\Delta(L_1)\times\Delta(L_2),\quad G_{\varphi_1+\varphi_2}(x+y)\geqslant G_{\varphi_1}(x)+G_{\varphi_2}(y).\]
\end{prop}
\begin{proof}
Let $(t_1,t_2)\in\mathbb R^2$, $n\in\mathbb N_{\geqslant 1}$ and $(s_1,s_2)\in V_{n}^{\varphi_1,t_1}(L_1)\times V_{n}^{\varphi_2,t_2}(L_2)$. By definition, for any $\varepsilon>0$ there exists $N\in\mathbb N_{\geqslant 1}$ such that 
\[\forall\,i\in\{1,2\},\quad s_i^N\in\mathcal F^{nN(t_i-\varepsilon)}(V_{nN}(L_i),\xi_{nN\varphi_i}).\]
By \cite[Corollary 5.6.2]{CMArakelovAdelic} (see also \cite[Remark C.3]{CMHS}), we obtain that 
\[\begin{split}-\ln\|(s_1s_2)^N&\|_{nN(\varphi_1+\varphi_2)}\geqslant nN(t_1+t_2-2\varepsilon)\\
&-\frac{3}{2}\nu(\Omega_\infty)\Big(\ln\dim_K(V_{nN}(L_1))+\ln\dim_K(V_{nN}(L_2))\Big).
\end{split}\]
Dividing the two sides of the equality by $N$ and taking the limit when $N\rightarrow+\infty$, we obtain 
\[-\ln\|s_1s_2\|_{n(\varphi_1+\varphi_2),\operatorname{sp}}\geqslant n(t_1+t_2-2\varepsilon).\] 
Since $\varepsilon>0$ is arbitrary, we obtain
\[-\|s_1s_2\|_{n(\varphi_1+\varphi_2),\operatorname{sp}}\geqslant n(t_1+t_2).\]
Therefore, one has 
\[V_n^{\varphi_1,t_1}(L_1)\cdot V_n^{\varphi_2,t_2}(L_2)\subseteq V_n^{\varphi_1+\varphi_2,t_1+t_2}(L_1\otimes L_2),\]
which implies
\[\Delta(V_\sbullet^{\varphi_1,t_1}(L_1))+\Delta(V_\sbullet^{\varphi_2,t_2}(L_2))\subseteq \Delta(V_\sbullet^{\varphi_1+\varphi_2,t_1+t_2}(L_1\otimes L_2)).\]
Let $(x,y)$ be an element of $\Delta(V_\sbullet(L_1))\times\Delta(V_{\sbullet}(L_2))$. For any $\varepsilon>0$ and any $(t_1,t_2)\in\mathbb R^2$ such that $t_1\leqslant G_{\varphi_1}(x)-\varepsilon$ and $t_2\leqslant G_{\varphi_2}(y)-\varepsilon$, one has 
\[(x,y)\in\Delta(V_\sbullet^{\varphi_1,t_1}(L_1))\times\Delta(V_\sbullet^{\varphi_2,t_2}(L_2))\]
and hence $x+y\in \Delta(V_\sbullet^{\varphi_1+\varphi_2,t_1+t_2}(L_1\otimes L_2))$. We thus obtain
\[G_{\varphi_1+\varphi_2}(x+y)\geqslant t_1+t_2-2\varepsilon.\]
Since $t_1$, $t_2$ and $\varepsilon$ are arbitrary, we deduce
\[G_{\varphi_1+\varphi_2}(x+y)\geqslant G_{\varphi_1}(x)+G_{\varphi_2}(x).\]
\end{proof}

\begin{coro}\label{Cor: concavity of vol chi}
Let $L$ be a slope-bounded invertible $\mathcal O_X$-module, $\varphi_1$ and $\varphi_2$ be elements of $\mathscr M(L)$, and $\delta\in[0,1]$. Then the following inequality holds
\begin{equation}\label{Equ: concavity of vol chi}\widehat{\mathrm{vol}}_{\chi}(L,\delta\varphi_1+(1-\delta)\varphi_2)\geqslant \delta\operatorname{\widehat{\mathrm{vol}}}_{\chi}(L,\varphi_1)+(1-\delta)\operatorname{\widehat{\mathrm{vol}}_{\chi}}(L,\varphi_2).\end{equation}
In other words, the function from $\mathscr M(L)$ to $\mathbb R$ sending $\varphi$ to $\widehat{\mathrm{vol}}_{\chi}(L,\varphi)$ is concave.
\end{coro}
\begin{proof}
We first treat the case where $\delta$ is a rational number. Let $k$ and $N$ be positive integers such that $N>k$. By 
\eqref{Equ: concave transform homogeneity}, \eqref{Equ: integral G is equal to vol chi} and Proposition \ref{Pro: super additive G}, we obtain 
\[\frac{\widehat{\mathrm{vol}}_{\chi}(L^{\otimes N},k\varphi_1+(N-k)\varphi_2)}{\operatorname{vol}(L^{\otimes N})}\geqslant\frac{\widehat{\mathrm{vol}}_{\chi}(L^{\otimes k},k\varphi_1)}{\operatorname{vol}(L^{\otimes k})}+\frac{\widehat{\mathrm{vol}}_{\chi}(L^{\otimes(N-k)},(N-k)\varphi_2)}{\operatorname{vol}(L^{\otimes (N-k)})},\]
or equivalently,
\[N\operatorname{\widehat{\mathrm{vol}}_{\chi}}(L,\textstyle{\frac{k}{N}\varphi_1+\frac{N-k}{N}\varphi_2})\geqslant k\operatorname{\widehat{\mathrm{vol}}_{\chi}}(L,\varphi_1)+(N-k)\operatorname{\widehat{\mathrm{vol}}_{\chi}}(L,\varphi_2).\]
Therefore the inequality \eqref{Equ: concavity of vol chi} holds in the case where $\delta$ is rational. The general case follows from the rational case together with the following estimate
\[\forall\,(\varphi,\psi)\in\mathscr M(L)^2,\quad \Big|\widehat{\mathrm{vol}}(L,\varphi)-\widehat{\mathrm{vol}}(L,\psi)\Big|\leqslant (d+1)\operatorname{vol}(L)\int_{\Omega}\sup_{x\in X_\omega^{\mathrm{an}}}|(\varphi-\psi)(x)|\,\nu(\mathrm{d}\omega).\]  
\end{proof}

\subsection{G\^ateaux differentiability}

\begin{defi}
Let $L$ be a big invertible $\mathcal O_X$-module. Let $\varphi$ and $\psi$ be two elements of $\mathscr M(L)$. For any $\omega\in\Omega$, we denote by $\operatorname{vol}(L_\omega,\varphi_\omega,\psi_\omega)$ the relative volume of $L_\omega$ with respect to the metric pair $(\varphi_\omega,\psi_\omega)$, which is defined as 
\[-\lim_{n\rightarrow+\infty}\frac{(d+1)!}{n^{d+1}}\ln\frac{\|\ndot\|_{n\varphi_\omega,\det}}{\|\ndot\|_{n\psi_\omega,\det}}.\]
We refer to \cite[Theorem 4.5]{MR3312451} for the convergence of the sequence definition the relative volume.
\end{defi}

\begin{prop}\label{Pro: local volume}
Let $L$ be a semi-ample and big  invertible $\mathcal O_X$-module and $(\varphi,\psi)\in\mathscr M(L)^2$. The following equality holds:
\[\widehat{\mathrm{vol}}_\chi(L,\varphi)-\widehat{\mathrm{vol}}_\chi(L,\psi)=\int_{\Omega}\operatorname{vol}(L_\omega,\varphi_\omega,\psi_\omega)\,\nu(\mathrm{d}\omega).\]
\end{prop}
\begin{proof}
For any positive integer $n$, let $\alpha_n$ be a non-zero element of $\det H^0(X,L^{\otimes n})$. By definition
\[\widehat{\mathrm{vol}}_\chi(L,\varphi)-\widehat{\mathrm{vol}}_\chi(L,\psi)=-\lim_{n\rightarrow+\infty}\frac{(d+1)!}{n^{d+1}}\int_{\omega\in\Omega}\ln\frac{\|\alpha_n\|_{n\varphi_\omega,\det}}{\|\alpha_n\|_{n\psi_\omega,\det}}\,\nu(\mathrm{d}\omega).\]
Note that 
\[\frac{1}{n\dim_K(H^0(X,L^{\otimes n}))}\bigg|\ln\frac{\|\ndot\|_{n\varphi_\omega,\det}}{\|\ndot\|_{n\psi_\omega,\det}}\bigg|\leqslant\sup_{x\in X_\omega^{\mathrm{an}}}|\varphi_\omega-\psi_\omega|(x).\]
By dominated convergence theorem we obtain
\[\begin{split}\widehat{\mathrm{vol}}_\chi(L,\varphi)-\widehat{\mathrm{vol}}_\chi(L,\psi)&=-\int_{\omega\in\Omega}\lim_{n\rightarrow+\infty}\frac{(d+1)!}{n^{d+1}}\ln\frac{\|\ndot\|_{n\varphi_\omega,\det}}{\|\ndot\|_{n\psi_\omega,\det}}\,\nu(\mathrm{d}\omega)\\
&=\int_{\Omega}\operatorname{vol}(L_\omega,\varphi_\omega,\psi_\omega)\,\nu(\mathrm{d}\omega).
\end{split}\]
\end{proof}

\begin{prop}\label{Pro: differentiability of vol chi} Let $\overline L=(L,\varphi)$ be an adelic line bundle on $X$. We assume that $L$ is semi-ample and big and that $\varphi$ is semi-positive.
The function $\widehat{\mathrm{vol}}_\chi(\ndot)$ on $\widehat{\mathrm{Pic}}_A(X)$ is G\^ateaux differentiable at $\overline L$ along the directions of $\mathscr M(\mathcal O_X)$. Moreover, for any $f\in\mathscr M(\mathcal O_X)$, the function 
\[(\omega\in\Omega)\longmapsto\int_{X_\omega^{\mathrm{an}}}f_\omega \,c_1(L_\omega,\varphi_\omega)^d\]
is $\nu$-integrable, and 
one has 
\[\lim_{t\rightarrow 0}\frac{\widehat{\mathrm{vol}}_\chi(\overline L(tf))-\widehat{\mathrm{vol}}_\chi(\overline L)}{t}=(d+1)\int_{\Omega}\nu(\mathrm{d}\omega)\int_{X_\omega^{\mathrm{an}}}f_\omega\, c_1(L_\omega,\varphi_\omega)^d.\] 
\end{prop}
\begin{proof}
By \cite[Theorem 1.2]{BGF2020b} and \cite[Theorem B]{MR2657428}, for any $\omega\in\Omega$, one has 
\[\lim_{t\rightarrow 0}\frac{\operatorname{vol}(L_\omega,\varphi_\omega+tf_\omega,\varphi_\omega)}{t}=(d+1)\int_{X_\omega^{\mathrm{an}}}f_\omega\,c_1(L_\omega,\varphi_\omega)^d.
\] 
Note that 
\[\Big|\frac{\operatorname{vol}(L_\omega,\varphi_\omega+tf_\omega,\varphi_\omega)}{t}\Big|\leqslant (d+1)\deg_{L}(X)\sup_{x\in X_\omega^{\mathrm{an}}}|f_\omega|(x).\]
Since the function
\[(\omega\in\Omega)\longmapsto\sup_{x\in X_\omega^{\mathrm{an}}}|f_\omega|(x) \]
is integrable, by Lebesgue's dominated convergence theorem we obtain, by using Proposition \ref{Pro: local volume}, that  
\[\lim_{t\rightarrow 0}\frac{\widehat{\mathrm{vol}}_\chi(\overline L(tf))-\widehat{\mathrm{vol}}_\chi(\overline L)}{t}=(d+1)\int_{\Omega}\nu(\mathrm{d}\omega)\int_{X_\omega^{\mathrm{an}}}f_\omega\,c_1(L_\omega,\varphi_\omega)^d.\]
\end{proof}

\begin{rema}
We conjecture that any big invertible $\mathcal O_X$-module $L$ is slope-bounded. If this is true, then for any metric family $\varphi\in\mathscr M(L)$, the $\chi$-volume $\widehat{\operatorname{vol}}(L,\varphi)$ takes real values. Hence the results of Propositions \ref{Pro: local volume} and \ref{Pro: differentiability of vol chi} hold without semi-amplitude assumption on $L$. Correspondingly, we conjecture that Theorem \ref{theorem:equidistribution} also holds when $L$ is only nef and big.
\end{rema}

\begin{coro}\label{Cor: several nef case}
Let $( M_1,\psi_1),\ldots,( M_d,\psi_d)$ be relatively nef adelic line bundles. For any $f\in\mathscr{C}^0_{\mathrm{a}}(X)$, the function
\begin{equation}\label{Equ: nu integrable multi line bundle}(\omega\in\Omega)\longmapsto\int_{X_\omega^{\mathrm{an}}}f_\omega\,c_1(M_{1,\omega},\psi_{1,\omega})\cdots c_1(M_{d,\omega},\psi_{d,\omega})\end{equation}
is $\nu$-integrable.
\end{coro}
\begin{proof}
By the multi-linearity of Monge-Amp\`ere measure, we may assume without loss of generality that all adelic line bundles $(M_i,\psi_i)$ are equal to the same one $(M,\psi)$ (c.f. \cite[Proposition~1.1.4]{CMIntersection}). Let $(L,\varphi)$ be a relatively ample adelic line bundle on $X$. By Proposition \ref{Pro: differentiability of vol chi}, for any $n\in\mathbb N_{\geqslant 1}$, the function
\[(\omega\in\Omega)\longmapsto \frac{1}{n^d}\int_{X_{\omega}^{\mathrm{an}}}f_\omega\,c_1(M_\omega^{\otimes n}\otimes L_\omega,n\psi_\omega+\varphi_\omega)^d\]
is $\mathcal A$-measurable. Passing to limit when $n\rightarrow+\infty$, we obtain the $\mathcal A$-measurability of the function
\[(\omega\in\Omega)\longmapsto\int_{X_\omega^{\mathrm{an}}}f_\omega\,c_1(M_\omega,\psi_\omega)^d.\]
Finally, since $f\in\mathscr C^0_{\mathrm{a}}(X)$, by \cite[Proposition 6.1.12]{CMArakelovAdelic}, we obtain that there exist two $\nu$-integrable functions $A_1$ and $A_2$ on $\Omega$ such that
\[\forall\,x\in X_\omega^{\mathrm{an}},\quad A_1(\omega)\leqslant f_\omega(x)\leqslant A_2(\omega).\]
Since each $c_1(M_\omega,\psi_\omega)^{d}$ has measure $c_1(M)^d$, the function
\[(\omega\in\Omega)\longmapsto\int_{X_\omega^{\mathrm{an}}}f_\omega\,c_1(M_\omega,\psi_\omega)^d\] 
is $\nu$-integrable. 
\end{proof}

\subsection{Measurability of fiber integrals}

\begin{defi}\label{def:Borel:measure:family}
 Let $\Omega'$ be an element of $\mathcal A$.
As \emph{Borel measure family} on $X$ over $\Omega'$,
we refer to a family $\eta=(\eta_\omega)_{\omega\in\Omega'}$, where each $\eta_\omega$ is a Borel measure on $X_\omega$, such that, for any $f = (f_{\omega})_{\omega \in \Omega'}\in\mathscr{C}^0_{\mathrm{a}}(X; \Omega')$, the function 
\[(\omega\in\Omega')\longmapsto\int_{X_\omega^{\mathrm{an}}}f_\omega(x)\,\eta_\omega(\mathrm{d}x)\]
is $\mathcal A|_{\Omega'}$-measurable and integrable with respect to the restriction of the measure $\nu$ to $\Omega'$. We denote by 
$\eta(f)$ the integral
\[\int_{\Omega'}\int_{X_\omega^{\mathrm{an}}}f_\omega(x)\,\eta_\omega(\mathrm{d}x)\,\nu(\mathrm{d}\omega).\]
\end{defi} 

\begin{rema}\label{Rem: finitenesss of measure}
Let $\eta=(\eta_\omega)_{\omega\in\Omega}$ be a Borel measure family. For any $\nu$-integrable function $A:\Omega\rightarrow\mathbb R$ which vanishes on $\Omega\setminus\Omega'$, we consider the family $f_A=(f_{A,\omega})_{\omega\in\Omega}$, where $f_{A,\omega}$ denotes the constant function on $X_\omega^{\mathrm{an}}$ taking value $A(\omega)$. 
We then obtain that the function 
\[(\omega\in\Omega')\longmapsto A(\omega)\eta_\omega(X_\omega^{\mathrm{an}})=\int_{x\in X_\omega^{\mathrm{an}}}f_{A,\omega}(x)\,\eta_\omega(\mathrm{d}x)\]
is $\nu$-integrable. This observation shows that the function 
\[(\omega\in\Omega')\longmapsto \eta_\omega(X_\omega^{\mathrm{an}})\]
is essentially bounded, namely there exists $C>0$ such that 
\[\{\omega\in\Omega'\,:\,\eta_\omega(X_\omega^{\mathrm{an}})>C\}\]
is a zero measure set.
\end{rema}

\begin{exem}\label{Exe: delte L}
Let $\overline L=(L,\varphi)$ be a relatively ample adelic line bundle on $X$, namely $L$ is an ample invertible $\mathcal O_X$-module and $\varphi=(\varphi_\omega)_{\omega\in\Omega}$ is a measurable and dominated family of semi-positive metrics. Let $Y$ be a reduced closed subscheme. Denote by $\delta_{\overline L,Y,\Omega'}=(\delta_{\overline L,Y,\omega})_{\omega\in\Omega'}$ the Borel measure family on $X$ over $\Omega'$ defined as follows: for any $\omega\in\Omega$, and any positive Borel function $f_\omega$ on $X_\omega^{\mathrm{an}}$,
\[\int_{X_\omega^{\mathrm{an}}}f_\omega(x)\,\delta_{\overline L,Y,\omega}(\mathrm{d}x):=\frac{1}{\deg_L(Y)}\int_{Y_\omega^{\mathrm{an}}}f_\omega(y)\,c_1(L_\omega|_{Y_\omega},\varphi_\omega|_{Y_\omega})^{\dim(Y)}(\mathrm{d}y).\]
This is a Borel probability measure on $X_\omega^{\mathrm{an}}$, which is supported on $Y_\omega^{\mathrm{an}}$. In the case where $Y$ is a closed point, the measure $\delta_{\overline L,Y,\omega}$ is given by the weighted average on points of $X_\omega^{\mathrm{an}}$. 
We refer to Proposition \ref{Pro: differentiability of vol chi} for the integrability of the function 
\[(\omega\in\Omega)\longmapsto\int_{X_\omega^{\mathrm{an}}}f_\omega(x)\,\delta_{\overline L,Y,\omega}(\mathrm{d}x)\]
when $f\in\mathscr C^0_{\mathrm{a}}(X)$ (so that the restriction of $f$ to $Y$ belongs to $\mathscr C^0_{\mathrm{a}}(Y)$). 
In the case where $S$ is proper and $\Omega'=\Omega$, one
can also interpret the expression $\delta_{\overline L,Y,\Omega}$ in terms of the arithmetic intersection theory. For any $f\in \mathscr C^0_{\mathrm{a}}(X)$ and any $t\in\mathbb R$, we denote by $\overline L(tf)$ the adelic line bundle $(L,\varphi+\mathrm{e}^{-tf})$. Then the following equality holds:
\[\delta_{\overline L,Y,\Omega'}(f)=\lim_{t\rightarrow 0}\frac{(\overline L(tf)|_Y^{\dim(Y)+1})_S-(\overline L^{\dim(Y)+1}|_Y)_S}{(\dim(Y)+1)\deg_L(Y)t},\]
provided that $(\mathcal O_X, \mathrm{e}^{-f})$ formes an integrable adelic line bundle.
\end{exem}

\begin{exem}
Let $(M_1,\psi_1),\ldots,(M_d,\psi_d)$ be relatively nef adelic line bundles on $X$. For any $\omega\in\Omega$, let 
\[\eta_\omega=c_1(M_{1,\omega},\psi_{1,\omega})\cdots c_1(M_{d,\omega},\psi_{d,\omega}).\]
By Corollary \ref{Cor: several nef case}, $(\eta_\omega)_{\omega\in\Omega}$ forms a Borel measure family on $X$ over $\Omega$.
\end{exem}

\begin{prop}\label{Pro: measurability logarithmic}
Let $\Omega'$ be an element of $\mathcal A$ and $\eta=(\eta_\omega)_{\omega\in\Omega'}$ be a Borel measure family on $X$ over $\Omega'$. Let $\overline M=(M,\psi)$ be an adelic line bundle on $X$ and $s$ be a non-zero global section of $M$. Then the function 
\[(\omega\in\Omega')\longmapsto\int_{X_{\omega}^{\mathrm{an}}}(-\ln|s|_{\psi_\omega}(x))\,\eta_{\omega}(\mathrm{d}x)\]
is $\mathcal A$-measurable and is bounded from below by an integrable function.
\end{prop}
\begin{proof}

Let $\Omega_0$ be the set of $\omega\in\Omega$ such that $|\ndot|_\omega$ is trivial. Then 
\[\Omega\setminus\Omega_0=\bigcup_{a\in K\setminus\{0\}}\{\omega\in\Omega\,:\,|a|_{\omega}\neq 1\}\]
is $\sigma$-finite (namely a countable union of elements of $\mathcal A$ which have a finite measure). Moreover, by \cite[Proposition 6.1.12]{CMArakelovAdelic}, a comparison with the trivial metric family over $\Omega_0$ shows that the function 
\[(\omega\in\Omega_0)\longmapsto \sup_{x\in X_\omega^{\mathrm{an}}}\big|\ln|s|_{\psi_\omega}(x)\big|\]
is integrable. Therefore the set
\[\Omega_{0,s}=\{\omega\in\Omega_0\,:\,\text{$|s|_{\psi_\omega}$ is not identically $1$}\}\]
is $\sigma$-finite. Therefore, we may choose a non-negative $\nu$-integrable function $A$ on $\Omega$ such that $A(\omega)>0$ for any $\omega\in(\Omega\setminus\Omega_0)\cup\Omega_{0,s}$. In fact, if we write $(\Omega\setminus\Omega_0)\cup\Omega_{0,s}$ as a countable union $\bigcup_{n\in\mathbb N}B_n$, where each $B_n$ is an element of finite measure in $\mathcal A$, then, with arbitrary choices of positive real numbers $b_n$ such that $b_n\nu(B_n)\leqslant 2^{-n}$, $n\in\mathbb N$, the following function is $\nu$-integrable
\[\sum_{n\in\mathbb N}b_n\indic_{B_n}\]
and vanishes nowhere on $(\Omega\setminus\Omega_0)\cup\Omega_{0,s}$.

For any $t>0$ and any $\omega\in\Omega$, let $f_{t,\omega}:X_\omega^{\mathrm{an}}\rightarrow\mathbb R$ be the function defined as follows:
\[f_{t,\omega}(x):=\min\{-\ln|s|_{\psi_\omega}(x),tA(\omega)\}.\]
This is a continuous function on $X_\omega^{\mathrm{an}}$, which yields a continuous metric $\mathrm{e}^{-f_\omega}$ on $\mathcal O_{X_\omega}$. Moreover, the function $f_{t,\omega}$ is bounded from below by $\min \{ -\ln\|s\|_{\psi_\omega}, tA(\omega) \}$ and bounded from above by $tA(\omega)$. By Proposition \cite[Proposition 6.2.12]{CMArakelovAdelic}, the function
\[g_s:\Omega\rightarrow\mathbb R,\quad (\omega\in\Omega)\longmapsto -\ln\|s\|_{\psi_\omega} \]
is integrable. Therefore, the family $f_t=(f_{t,\omega})_{\omega\in\Omega}$ belongs to $\mathscr C^0_{\mathrm{a}}(X)$ and hence  the function
\[(\omega\in\Omega')\longmapsto \int_{X_\omega^{\mathrm{an}}}f_{t,\omega}(x)\,\eta_\omega(\mathrm{d}x)\]
is $\mathcal A$-measurable. Passing to limit when $t\rightarrow+\infty$, we obtain the measurability of the function 
\[(\omega\in\Omega')\longmapsto\int_{X_{\omega}^{\mathrm{an}}}(-\ln|s|_{\psi_\omega}(x))\,\eta_{\omega}(\mathrm{d}x).\] Finally, by definition, for any $\omega\in\Omega'$, one has 
\[\int_{X_\omega^{\mathrm{an}}}(-\ln|s|_{\psi_\omega}(x))\,\eta_\omega(\mathrm{d}x)\geqslant -\ln\|s\|_{\psi_\omega}\eta_{\omega}(X_\omega^{\mathrm{an}}).\]
Since the function $(\omega\in\Omega')\mapsto\eta_\omega(X_\omega^{\mathrm{an}})$ is essentially bounded, the second assertion is true.
\end{proof}

\begin{prop}\label{Pro: integrability on an open set}
Let $U$ be a non-empty Zariski open subset of $X$, and $f=(f_\omega)_{\omega\in\Omega}$ be a measurable family, where each $f_\omega$ is a continuous function on $U_{\omega}^{\mathrm{an}}$, such that there exists a $\nu$-integrable function $g:\Omega\rightarrow\mathbb R$ satisfying
\[\forall\,\omega\in\Omega,\;\forall\,x\in X_\omega^{\mathrm{an}},\quad f_\omega(x)\geqslant g(\omega).\]
Let $\Omega'$ be a $\sigma$-finite element of $\mathcal A$ and $(\eta_\omega)_{\omega\in\Omega}$ be a Borel measure family on $X$ over $\Omega'$. Assume that, there exist an adelic line bundle $(M,\psi)$ and a non-zero section $s\in H^0(X,M)$ such that the non-vanishing locus of $s$ is contained in $U$ and that 
\[\int_{X_\omega^{\mathrm{an}}}(-\ln|s|_{\psi_\omega}(x))\,\eta_\omega(\mathrm{d}x)<+\infty\quad\text{ $\nu$-almost everywhere on $\Omega'$}.\] 
Then the function 
\[(\omega\in\Omega)\longmapsto \int_{U_\omega^{\mathrm{an}}}f_\omega(x)\,\eta_\omega(\mathrm{d}x) \]
is $\mathcal A$-measurable.
\end{prop}
\begin{proof}
We choose a non-negative $\nu$-integrable function $A$ on $\Omega$ such that $A(\omega)>0$ for any $\omega\in \Omega'$. This is possible since $\Omega'$ is $\sigma$-finite. Moreover, without loss of generality, we may assume (by Remark \ref{Rem: finitenesss of measure} and the condition of the proposition on $(\eta_\omega)_{\omega\in\Omega}$) that 
\[\int_{X_\omega^{\mathrm{an}}}(-\ln|s|_{\psi_\omega}(x))\,\eta_\omega(\mathrm{d}x)\in\mathbb R\]
for any $\omega\in\Omega'$.

For any $t>0$ and any $\omega\in\Omega$, let $f_{t,\omega}:X_\omega^{\mathrm{an}}\rightarrow\mathbb R$ be the function defined as
\[f_{t,\omega}(x):=\min\{f_\omega(x)-\ln|s|_{\psi_\omega}(x),tA(\omega)\},\]
where by convention $f_{t,\omega}(x)=tA(\omega)$ when $s(x)=0$. Since $f_\omega(x)$ is continuous and $-\ln|s|_{\psi_{\omega}}(x)$ tends to $+\infty$ when $x$ tends to some point $x_0\in X_\omega^{\mathrm{an}}$ such that $s(x_0)=0$, we obtain that the function $f_{t,\omega}$ is continuous on $X_\omega^{\mathrm{an}}$. Moreover, the function $f_{t,\omega}$ is bounded from above by $tA(\omega)$ and bounded from below by \[\min \{ g(\omega)-\ln\|s\|_{\psi_\omega}, tA(\omega) \}.\] Therefore the family $(f_{t,\omega})_{\omega\in\Omega}$ belongs to $\mathscr C^0_{\mathrm{a}}(X)$. Hence by Proposition \ref{Pro: measurability logarithmic} we obtain the measurability of the function
\[(\omega\in\Omega)\longmapsto\int_{X_\omega^{\mathrm{an}}}(f_\omega(x)-\ln|s|_{\psi_\omega}(x))\,\eta_\omega(\mathrm{d}x).\]
\end{proof}

\begin{rema}
Let $(L,\varphi)$ be a relatively nef adelic line bundle. For any $\omega\in\Omega$, let $\eta_\omega=c_1(L_\omega,\varphi_\omega)^d$. Then, for any adelic line bundle $(M,\psi)$ and any non-zero section $s\in H^0(X,M)$, one has 
\[\forall\,\omega\in\Omega,\quad \int_{X_\omega^{\mathrm{an}}}(-\ln|s|_{\psi_\omega}(x))\,\eta_\omega(\mathrm{d}x)\in\mathbb R.\]
We refer to \cite{MR2543659} for the non-Archimedean case.
\end{rema}

\section{Global adelic space}

In this section, we fix a 
adelic curve $S=(K,(\Omega,\mathcal A,\nu),\phi)$ such that $K$ is countable and perfect.  
Let $\pi:X\rightarrow\Spec K$ be an integral projective scheme.  Denote by $K(X)$ the field of rational functions on $X$. This is also a countable field.

\subsection{Adelic measure space}\label{subsec:adelic:measure:space}

Let $X$ be a projective scheme over $K$ and $\Omega'$ be an element of $\mathcal A$. 
We denote by $X^{\mathrm{an}}_{\Omega'}$ the disjoint union $\coprod_{\omega\in\Omega'}X_\omega^{\mathrm{an}}$. Denote by $\pi:X_{\Omega'}^{\mathrm{an}}\rightarrow\Omega'$ the map sending the elements of $X_\omega^{\mathrm{an}}$ to $\omega$. For any Zariski open subset $U$ of $X$, let $U_{\Omega'}^{\mathrm{an}}$ be the disjoint union $\coprod_{\omega\in\Omega'}U_\omega^{\mathrm{an}}$.

\begin{defi}\label{def:sigma:algebra:adelic:measure:space}
We equip $X_{\Omega'}^{\mathrm{an}}$ with the smallest $\sigma$-algebra $\mathcal B_{X,\Omega'}$ which satisfies the following conditions:
\begin{enumerate}[label=\rm(\arabic*)]
\item the map $\pi:X_{\Omega'}^{\mathrm{an}}\rightarrow \Omega'$ is measurable,
\item for any Zariski open subset $U$ of $X$, the set $U_{\Omega'}^{\mathrm{an}}$ belongs to $\mathcal B_{X,\Omega'}$,
\item 
for any Zariski open subset $U$ of $X$ and any measurable family $f$ of continuous functions over $U$ (i.e. $f \in \mathscr{C}^0_{\mathrm{m}}(U)$),
the function
$f_{\Omega'}$ on $U_{\Omega'}^{\mathrm{an}}$ defined as
\[\forall\,\omega\in\Omega',\;\forall\,x\in U_\omega^{\mathrm{an}},\quad f_{\Omega'}(x):=f_{\omega}(x)\]
is $\mathcal B_{X,\Omega'}|_{U_{\Omega'}^{\mathrm{an}}}$-measurable.
\end{enumerate} 
\end{defi}

\begin{rema}
The above third condition can be replaced by the following (3)':
\begin{enumerate}[label=\rm(\arabic*)']\setcounter{enumi}{2}
\item for any Zariski open subset $U$ of $X$, any  measurable\footnote{We refer to \cite[Definition 6.1.27]{CMArakelovAdelic} for the measurability of metric family. Note that the projectivity of the scheme is assume there. However, the definition can be easily extended to the case of any scheme of finite type over the adelic curve.}  metric family $\varphi$ on $\mathcal O_U$, and any regular function $b$ on $U$, the function $|b|_{\varphi,\Omega'}$ on $U_{\Omega'}^{\mathrm{an}}$ defined as
\[\forall\,\omega\in\Omega',\;\forall\,x\in U_\omega^{\mathrm{an}},\quad |b|_{\varphi,\Omega'}(x):=|b|_{\varphi_\omega}(x)\]
is $\mathcal B_{X,\Omega'}|_{U_{\Omega'}^{\mathrm{an}}}$-measurable.
\end{enumerate}
\end{rema}

\begin{rema}\label{Rem: mesurability abs value}
Let $U$ be a  Zariski open subset of $X$. We consider the trivial metric family on $\mathcal O_U$. Then the point (3) in the above definition shows that, for any regular function $f$ on $U$, the function $|f|_{\Omega'}:U_{\Omega'}^{\mathrm{an}}\rightarrow\mathbb R_{\geqslant 0}$, which sends $x\in U_\omega^{\mathrm{an}}$ to $|f|_\omega(x)$, is $\mathcal B_{X,\Omega'}|_{U_{\Omega'}^{\mathrm{an}}}$-measurable.

Assume that the scheme $X$ is integral. Let $q$   be a rational function on $X$ and $U$ be the maximal open subscheme over which the rational function $q$ is defined. We consider the function $|q|_{\Omega'}$ on $X_{\Omega'}^{\mathrm{an}}$ sending $x\in X_\omega^{\mathrm{an}}$ to $|q|_\omega(x)$. Note that, on the Zariski open subset $U$, the rational function $q$ coincides with a regular function $b$ on $U$. Moreover, the following equality holds
\[|q|_{\Omega'}(x)=\begin{cases}
|b|_{\Omega'}(x),&\text{if $x\in U_{\Omega'}^{\mathrm{an}}$},\\
+\infty,&\text{if $x\in X_{\Omega'}^{\mathrm{an}}\setminus U_{\Omega'}^{\mathrm{an}}$.}
\end{cases}\]
In particular, the function $|q|_{\Omega'}$ is $\mathcal B_{X,\Omega'}$-measurable.
\end{rema}

\begin{prop}\label{Pro: measurability of local interals}
Let $f:X_{\Omega'}^{\mathrm{an}}\rightarrow \mathbb R_{\geqslant 0}$ be a $\mathcal B_{X,\Omega'}$-measurable function. 
\begin{enumerate}[label=\rm(\arabic*)]
\item For any $\omega\in\Omega'$, $f|_{X_\omega^{\mathrm{an}}}$ is a Borel measurable function.
\item Let $\eta=(\eta_{\omega})_{\omega\in\Omega'}$ be a Borel measure  family on $X$ over $\Omega'$. Then the function 
\[(\omega\in\Omega')\longmapsto\int_{X_\omega^{\mathrm{an}}}f(x)\,\eta_{\omega}(\mathrm{d}x)\]
is $\mathcal A|_{\Omega'}$-measurable.
\end{enumerate}
\end{prop}
\begin{proof}
If $U$ is a Zariski open subset of $X$, and $\varphi$ is a measurable metric family on $\mathcal O_U$, and $b$ is a regular function on $U$, we extend the domain of definition of $|b|_{\varphi,\Omega'}$ to $X_{\Omega'}^{\mathrm{an}}$ by letting $|b|_{\varphi,\Omega'}(x)=0$ for $x\in X_{\Omega'}^{\mathrm{an}}\setminus U_{\Omega'}^{\mathrm{an}}$. In this way we can consider $|b|_{\varphi,\Omega'}$ as a $\mathcal B_{X,\Omega'}$-measurable function on $X_{\Omega'}^{\mathrm{an}}$.

Let $\mathcal H_\eta$ be the set of bounded functions $f:X_{\Omega'}^{\mathrm{an}}\rightarrow\mathbb R$ which satisfies the condition predicted in the proposition, namely  $f|_{X_\omega^{\mathrm{an}}}$ is a Borel function for any $\omega\in\Omega'$, and the function
\[(\omega\in\Omega')\longmapsto\int_{X_\omega^{\mathrm{an}}}f(x)\,\eta_\omega(\mathrm{d}x)\]
is $\mathcal A|_{\Omega'}$-measurable. Note that $\mathcal H_{\eta}$ is a $\lambda$-family, namely
\begin{enumerate}[label=\rm(\roman*)]
\item the constant function $1$ belongs to $\mathcal H_\eta$;
\item if $f$ and $g$ are two elements of $\mathcal H_\eta$,  and $a$ and $b$ are non-negative numbers, then $af+bg\in\mathcal H_\eta$;
\item if $f$ and $g$ are two elements of $\mathcal H_\eta$ such that $f\leqslant g$, then $g-f\in\mathcal H_\eta$;
\item if $(f_n)_{n\in\mathbb N}$ is an increasing and uniformly bounded sequence of functions in $\mathcal H_\eta$, then the limit of the sequence $(f_n)_{n\in\mathbb N}$ belongs to $\mathcal H_\eta$.
\end{enumerate}

Let $\mathcal C$ be set of functions $X_{\Omega'}^{\mathrm{an}}\rightarrow\mathbb R$ of the form $\indic_{\pi^{-1}(A)}|b|_{\varphi,\Omega'}$, where $A$ is an element of $\mathcal A$ contained in $\Omega'$ and $b$ is a regular function on a Zariski open subset $U$ of $X$. Here we extend the domain of definition of $|b|_{\varphi,\Omega'}$ to $X_{\Omega'}^{\mathrm{an}}$ by letting $|b|_{\varphi,\Omega'}(x)=0$ for $x\in X_{\Omega'}^{\mathrm{an}}\setminus U_{\Omega'}^{\mathrm{an}}$.  Moreover, the $\sigma$-algebra $\mathcal B_{X,\Omega'}$ is equal to the $\sigma$-algebra $\sigma(\mathcal C)$ generated by $\mathcal C$. Note that, for any regular function $b$ on a Zariski open subset $U$ of $X$, for any $\omega\in\Omega'$ the function $|b|_{\varphi,\omega}$ is continuous on $U_\omega^{\mathrm{an}}$. By Proposition \ref{Pro: integrability on an open set}, the family $\mathcal C$ is contained in $\mathcal H$. If $A'$ is another element of $\mathcal A$ contained in $\Omega'$, $b'$ is a regular function on a Zariski open subset $U'$ of $X$ and $\varphi'$ is a measurable metric family of $\mathcal O_{U'}$, then one has 
\[(\indic_{\pi^{-1}(A)}|b|_{\varphi,\Omega'})(\indic_{\pi^{-1}(A')}|b'|_{\varphi',\Omega'})=\indic_{\pi^{-1}(A\cap A')}|bb'|_{\varphi+\varphi',\Omega'},\]
where wen consider $bb'$ as a regular function on $U\cap U'$ and $\varphi+\varphi'$ as a mesurable metric family of $\mathcal O_{U\cap U'}$. The function family $\mathcal C$ is hence stable by multiplication. By monotone class theorem $\mathcal H$ contains all bounded $\sigma(\mathcal C)$-measurable functions. Since any $\mathcal B_{X,\Omega'}$-measurable function can be written as a limit of bounded $\mathcal B_{X,\Omega'}$-measurable functions, the proposition is thus proved.
\end{proof}

\begin{prop}\label{Pro: measurability of green function}
Let $L$ be an invertible $\mathcal O_X$-module and $\varphi$ be a measurable metric family on $L$. For any Zariski open subset $U$ of $X$ and any section $s\in H^0(U,L)$, the function $|s|_{\Omega'}:U_{\Omega'}^{\mathrm{an}}\rightarrow\mathbb R$ sending $x\in U_\omega^{\mathrm{an}}$ to $|s|_\omega(x)$ is $\mathcal B_{X,\Omega'}|_{U_{\Omega'}^{\mathrm{an}}}$-measurable.
\end{prop} 
\begin{proof}
We choose a finite open covering $(U_i)_{i\in I}$ of $U$ such that, on each $U_i$, the invertible $\mathcal O_X$-module $L$ is isomorphic to $\mathcal O_{U_i}$. By definition (c.f. (3) in Definition~\ref{def:sigma:algebra:adelic:measure:space}), the restriction of $|s|_{\Omega'}$ to each $U_{i,\Omega'}^{\mathrm{an}}$ is measurable.
Hence $|s|_{\Omega'}$ is also measurable.
\end{proof}

\begin{defi}
Let  $\Omega'$ be an element of $\mathcal A$ and $\eta$ be a Borel measure family on $X$ over $\Omega'$. Proposition \ref{Pro: measurability of local interals} shows that, the map 
\[(B\in\mathcal B_{X,\Omega'})\longmapsto\int_{\omega\in\Omega'}\nu(\mathrm{d}\omega)\int_{X_{\omega}^{\mathrm{an}}}\indic_B(x)\,\eta_\omega(\mathrm{d}x)\]
defines a measure on the measurable space $(X_{\Omega'}^{\mathrm{an}},\mathcal B_{X,\Omega'})$. 
We denote by $\eta_{\Omega'}$ this measure. By abuse of notation, it is  often denoted by $\eta$. For any non-negative $\mathcal B_{X,\Omega'}$-measurable function $f:X_{\Omega'}^{\mathrm{an}}\rightarrow\mathbb R_{\geqslant 0}$, one has
\[\int_{X_{\Omega'}^{\mathrm{an}}}f(x)\,\eta_{\Omega'}(\mathrm{d}x)=\int_{\Omega'}\nu(\mathrm{d}\omega)\int_{X_{\omega}^{\mathrm{an}}}f(x)\,\eta_\omega(\mathrm{d}x).\] 
\end{defi}

\begin{rema}\label{Rem: unique determined by integrals along C}
Let $U$ be a dense Zariski open set of $X$. Then it is easy to see that $\eta_{\Omega'}(X_{\Omega'}^{\mathrm{an}} \setminus U_{\Omega'}^{\mathrm{an}}) = 0$. Moreover, if there exists an adelic line bundle $(M,\psi)$ such that $M$ is ample and that, for any $n\in\mathbb N_{\geqslant 1}$ and any non-zero section $s\in H^0(X,M^{\otimes n})$, one has 
\[\forall\,\omega\in\Omega,\quad\int_{X_\omega^{\mathrm{an}}}(-\ln|s|_{\psi_\omega}(x))\,\eta_\omega(\mathrm{d}x)\in\mathbb R,\]
then, viewed as a measure on $(X_{\Omega'},\mathcal B_{X,\Omega'})$, $\eta$ is uniquely determined by the integrals of functions in $\mathscr C_{\mathrm{a}}^0(X;\Omega')$. In other words, if $\eta'=(\eta'_\omega)_{\omega\in\Omega}$ is another Borel measure family on $X$ over $\Omega'$ such that 
\[\int_{\omega\in\Omega'}\nu(\mathrm{d}\omega)\int_{X_\omega^{\mathrm{an}}}f_\omega(x)\,\eta_\omega(\mathrm{d}x)=\int_{\omega\in\Omega'}\nu(\mathrm{d}\omega)\int_{X_\omega^{\mathrm{an}}}f_\omega(x)\,\eta_\omega'(\mathrm{d}x),\]
then, as measures on $(X_{\Omega'}^{\mathrm{an}},\mathcal B_{X,\Omega'})$, one has $\eta_{\Omega'}=\eta'_{\Omega'}$. This follows from the proofs of Propositions \ref{Pro: measurability logarithmic} and \ref{Pro: integrability on an open set}.
\end{rema}

\begin{prop}\label{prop:Radon-Nikodym}
Let  $\Omega'$ be an element of $\mathcal A$ and $\eta$ be a Borel measure family on $X$ over $\Omega'$. 
Let $p = (p_\omega)_{\omega \in \Omega'} \in \mathscr{C}_{\mathrm{m}}^0(X;\Omega')$.
We assume that $p\,\eta_{\Omega'} = \eta_{\Omega'}$ as measures on $(X_{\Omega'}^{\mathrm{an}},\mathcal B_{X,\Omega'})$, that is,
for any bounded measurable function  $u: X_{\Omega'}^{\mathrm{an}}\rightarrow\mathbb R$, one has 
\[\int_{X_{\Omega'}^{\mathrm{an}}}u(x)\,\eta_{\Omega'}(\mathrm{d}x)=\int_{X_{\Omega'}^{\mathrm{an}}}u(x)p(x)\,\eta_{\Omega'}(\mathrm{d}x),\]
where $p(x)=p_\omega(x)$ when $x\in X_\omega^{\mathrm{an}}$. 
Then there exists $\Omega'' \in \mathcal{A}$ such that $\Omega'' \subseteq \Omega'$, $\nu(\Omega' \setminus \Omega'') = 0$ and that
$p_\omega = 1$ holds $\eta_\omega$-almost everywhere on $X_{\omega}^{\mathrm{an}}$ for all $\omega \in \Omega''$.
\end{prop}
\begin{proof}
By our assumption, $p(x)=1$ holds $\eta_{\Omega'}$-almost everywhere. In particular, one has 
\[\int_{X^{\mathrm{an}}_{\Omega'}}\big|p(x)-1\big|\,\eta_{\Omega'}(\mathrm{d}x) =\int_{\Omega'} \nu(\mathrm{d}\omega)\int_{X_{\omega}^{\mathrm{an}}}\big|p_\omega(x)-1\big|\,\eta_{\omega}(\mathrm{d}x)=0.\]
Therefore, there exists $\Omega''\subseteq\Omega'$ such that $\nu(\Omega'\setminus\Omega'')=0$ and that the equality 
\[p_\omega(x)=1\]
holds $\eta_{\omega}$-almost everywhere for any $\omega\in
\Omega''$, as required.
\end{proof}


\subsection{Comparison between $\mathcal B_{X,\Omega'}$ and $\mathcal B'_{X,\Omega'}$}.
\subsubsection{Approximation of continuous functions}

Let $X$ be a projective reduced scheme over $\CC$.
Let $L$ be a line bundle on $X$ and $\varphi = \{ |\ndot|_{\varphi}(x) \}_{x \in X(\CC)}$ be a metric of $L$.
We say $\varphi$ is \emph{dominated from above by a continuous metric} if there exist a continuous metric $\psi$ of $L$ and
a function $f : X(\CC) \to \RR_{\geqslant 0}$ such that $\varphi = \psi + \mathrm{e}^{-f}$.
For $s \in H^0(X, L)$, we set $\|s\|_{\varphi} := \sup \{ |s|_{\varphi}(x) \mid x \in X(\CC) \}$.
Note that $\|s\| < \infty$ because $\|s\|_{\varphi} \leqslant \| s\|_{\psi}$.
Moreover, $\|\ndot\|_{\varphi}$ gives rise to a norm on $H^0(X, L)$.
We assume that $L$ is globally generated, that is,
$H^0(X, L) \otimes \OO_X \to L$ is surjective.
Then the quotient metric by $\|\ndot\|_{\varphi}$ is denoted by $\varphi_{\mathrm{FS}}$.
It is easy to see the following (c.f. \cite[the proof of Proposition~2.2.22 and Proposition~2.2.23]{CMArakelovAdelic}):
\begin{enumerate}[label=\rm(\alph*)]
\item $\varphi \leqslant \varphi_{\mathrm{FS}}$.
\item Let $M$ be another globally generated line bundle and $\varphi'$ be a metric dominated from above by a continuous metric.
Then $(\varphi + \varphi')_{\mathrm{FS}} \leqslant \varphi_{\mathrm{FS}} + \varphi'_{\mathrm{FS}}$.
\item If $\varphi$ is a quotient metric, then $\varphi = \varphi_{\mathrm{FS}}$.
\end{enumerate}

Let $L$ be a very ample line bundle on $X$ and $\varphi$ be a quotient metric by a Hermitian norm on $H^0(X, L)$.
Let $U$ be a dense Zariski open set of $X$ and $f : X(\CC) \to \RR_{\geqslant 0}$ be a function such that
$f$ is continuous over $U(\CC)$. Let $\psi_n = (n\varphi + \mathrm{e}^{-f})_{\mathrm{FS}}$ and
$\psi_n = n\varphi + \mathrm{e}^{-f_n}$.

\begin{prop}
The function $f_n$ is continuous on $X(\CC)$, $f_n \leqslant f$ and $f_n \leqslant f_{n+1}$.
\end{prop}

\begin{proof}
Since $\psi_n$ and $n\varphi$ are continuous on $X(\CC)$, so the first assertion is obvious.
By the above (a), $n \varphi + \exp(-f) \leqslant  \psi_n$, so the second assertion follows.
By using (b) and (c),
\begin{align*}
(n+1)\varphi + \exp(-f_{n+1}) & = ((n+1)\varphi + \exp(-f))_{\mathrm{FS}} \\
& = (\varphi + (n\varphi + \exp(-f)))_{\mathrm{FS}}
\leqslant \varphi + n\varphi + \exp(-f_n),
\end{align*}
as required.
\end{proof}

\begin{theo}\label{theorem:pontwise:limit}
For $x\in U(\CC)$, $\lim_{n\to\infty} f_n(x) = f(x)$.
\end{theo}

\begin{proof}
For $x \in X(\mathbb{C})$, let $V_x = \{ s \in H^0(X, L) \mid s(x) = 0 \}$. Then, as $L$ is very ample, 
$V_x$ is a codimension $1$ subspace of $H^0(X, L)$, and
$V_x \not= V_{y}$ for $x \not= y$. We assume that $x \in U(\CC)$.
Let $s_x$ be a non-zero vector in $H^0(X, L)$ which is orthogonal to $V_x$. 
For $y \in X(\mathbb{C}) \setminus \{x\}$, 
\[
|s_x|_\varphi(y) < \|s_x\| = |s_x|_\varphi(x).
\]
because $s_x$ is not orthogonal to $V_y$. Moreover, by definition,
\[
|s_x^{\otimes n}|_{\psi_n}(x) \leqslant \|s_x^{\otimes n} \|_{n\varphi + \mathrm{e}^{-f}} 
= \sup_{y \in X(\mathbb{C})} \mathrm{e}^{-f(y)} (|s_x|_{\varphi}(y))^n,
\]
which leads to
\[
\mathrm{e}^{-f_n(x)} = \frac{|s_x^{\otimes n}|_{\psi_n}(x)}{|s_x^{\otimes n}|_{n \varphi}(x)} \leqslant  \sup_{y \in X(\mathbb{C})} 
\mathrm{e}^{-f(y)} \left(\frac{|s_x|_{\varphi}(y)}{|s_x|_{\varphi}(x)}\right)^n
\]
for all $n \geqslant 1$.
Fix an arbitrary positive number $\varepsilon$.

\begin{clai}
For each $\xi \in X(\mathbb{C})$, we can find an open neighborhood $U_{\xi}$ of $\xi$ and
a positive integer $n_{\xi}$ such that
\[
\mathrm{e}^{-f(y)}  \left(\frac{|s_{x}|_{\varphi}(y)}{|s_{x}|_{\varphi}(x)}\right)^{n_\xi} \leqslant  \mathrm{e}^{-f(x)+\varepsilon}
\]
for all $y \in U_\xi$. 
\end{clai}

\begin{proof}
First we assume that $\xi = x$. Then the assertion is obvious.
Next we assume that  $\xi \not= x$.
Then we can find $n_{\xi} \in \ZZ_{\geqslant 1}$ such that
\[
\left(\frac{|s_{x}|_{\varphi}(\xi)}{|s_{x}|_{\varphi}(x)}\right)^{n_\xi} < \mathrm{e}^{-f(x)},
\]
so we can find an open neighborhood $U_{\xi}$ of $\xi$ such that
\[
\left(\frac{|s_{x}|_{\varphi}(y)}{|s_{x}|_{\varphi}(x)}\right)^{n_\xi} < \mathrm{e}^{-f(x)}
\]
for all $y \in U_{\xi}$. Therefore, for $y \in U_{\xi}$,
\[
 \mathrm{e}^{-f(y)} \left(\frac{|s_{x}|_{\varphi}(y)}{|s_{x}|_{\varphi}(x)}\right)^{n_\xi} \leqslant
 \left(\frac{|s_{x}|_{\varphi}(y)}{|s_{x}|_{\varphi}(x)}\right)^{n_\xi} < \mathrm{e}^{-f(x)} < \mathrm{e}^{-f(x)+\varepsilon},
\]
as required.
\end{proof}

Since $X(\mathbb{C})$ is compact and $X(\mathbb{C}) = \bigcup_{\xi \in X(\mathbb{C})} U_\xi$,
there are $\xi_1, \ldots, \xi_N$ such that $X(\mathbb{C}) = U_{\xi_1} \cup \cdots \cup U_{\xi_N}$.
Thus, for all $n \geqslant \max \{ n_{\xi_1}, \ldots, n_{\xi_N} \}$, one has
\[
 \mathrm{e}^{-f_n(x)} \leqslant \sup_{y \in X(\mathbb{C})} \mathrm{e}^{-f(y)}  \left(\frac{|s_{x}|_{\varphi}(y)}{|s_{x}|_{\varphi}(x)}\right)^{n} \leqslant  \mathrm{e}^{-f(x) + \varepsilon},
\]
and hence
\[
\lim_{n\to\infty} \mathrm{e}^{-f_n(x)} \leqslant  \mathrm{e}^{-f(x) + \varepsilon}.
\]
Thus one has $\lim_{n\to\infty} f_n(x) \geqslant f(x)$ because $\varepsilon$ is arbitrary.
Therefore, 
one obtains the assertion.
\end{proof}

\subsubsection{Measurability on $\Omega_{\infty}$}
Let $((\Omega, \mathcal{A}, \nu), \phi)$ be an adelic structure of a countable filed $K$.
Let $X$ be a geometrically integral projective variety over $K$ and $\Omega'$ be an element of $\mathcal A$.

\begin{defi}
We equip $X_{\Omega'}^{\mathrm{an}}$ the smallest $\sigma$-algebra $\mathcal B'_{X,\Omega'}$ which satisfies the following conditions:
\begin{enumerate}[label=\rm(\arabic*)]
\item the map $\pi:X_{\Omega'}^{\mathrm{an}}\rightarrow \Omega'$ is measurable,
\item for any Zariski open subset $U$ of $X$, the set $U_{\Omega'}^{\mathrm{an}}$ belongs to $\mathcal B'_{X,\Omega'}$,
\item for any Zariski open subset $U$ of $X$ and any regular function $b$ on $U$, the function $|b|_{\Omega'}$ on $U_{\Omega'}^{\mathrm{an}}$ defined as
\[\forall\,\omega\in\Omega',\;\forall\,x\in U_\omega^{\mathrm{an}},\quad |b|_{\Omega'}(x):=|b|_\omega(x)\]
is $\mathcal B'_{X,\Omega'}$-measurable.
\end{enumerate} 
Obviously $\mathcal B'_{X,\Omega'} \subseteq \mathcal B_{X,\Omega'}$.
\end{defi}

\begin{prop}\label{Pro: measurability of green function 02}
Let $L$ be a very ample invertible $\mathcal{O}_X$-module and $\varphi$ be a family of semipositive metrics of $L$ such that
$(L, \varphi)$ is measurable.
Then, for $s \in H^0(X, L)$,  the function $|s|_{\varphi, \Omega'}:X_{\Omega'}^{\mathrm{an}}\rightarrow\mathbb R$ sending $x\in X_\omega^{\mathrm{an}}$ to $|s|_{\varphi_\omega}(x)$ is  $\mathcal B'_{X,\Omega'}$-measurable.
\end{prop} 
\begin{proof}
We begin with the case where $\varphi$ is a quotient metric family. Let $E=H^0(X,L)$, $\xi$ be a norm family on $E$ such that $(E,\xi)$ forms a measurable vector bundle on $S$, and  $u:X\rightarrow\mathbb P(E)$ be the canonical projective $K$-morphism. Note that $u^*(\mathcal O_E(1))\cong L$. Assume that the metric family $\varphi$ is induced by $\xi$ and the morphism $u$. 
If $s = 0$, then the assertion is obvious, so we may assume that $s \not= 0$.
 Let $U$ be
a Zariski open set of $X$ given by $\{ \xi \in X \mid s(\xi) \not= 0 \}$.
For $t \in E \setminus \{ 0 \}$, we choose $\lambda \in K(X)^{\times}$ such that $t = \lambda s$.
Note that $\lambda$ is regular on $U$.
Let us consider $|\lambda|_{\Omega'}^{-1}(x)\cdot\|t\|_\omega$ on $U_{\Omega}^{\mathrm{an}}$.
For a point $x$ with $|\lambda|_{\Omega'}(x)=0$ (i.e. $t(x) = 0$), the value of $|\lambda|_{\Omega'}^{-1}(x)\cdot\|t\|_\omega$ is
defined to be $\infty$. Thus,
\[
\inf\limits_{\begin{subarray}{c}
t\in E\setminus\{0\},\;\lambda\in K(X)^{\times} \\
t=\lambda s
\end{subarray}}|\lambda|_{\Omega'}^{-1}(x)\cdot\|t\|_\omega = \inf\limits_{\begin{subarray}{c}
t\in E\setminus\{0\},\;\lambda\in K(X)^{\times} \\
t(x) \not= 0,\;t=\lambda s
\end{subarray}}|\lambda|_{\Omega'}^{-1}(x)\cdot\|t\|_\omega,
\]
and hence
\[|s|_{\varphi, \Omega'}(x)=\inf\limits_{\begin{subarray}{c}
t\in E\setminus\{0\},\;\lambda\in K(X)^{\times} \\
t=\lambda s
\end{subarray}}|\lambda|_{\Omega'}^{-1}(x)\cdot\|t\|_\omega\]
for all $\omega\in\Omega$ and $x\in U_\omega^{\mathrm{an}}$.
Therefore $|s|_{\varphi, \Omega'}$ is  $\mathcal B'_{X,\Omega'}$-measurable because
$|\lambda|_{\Omega'}^{-1}(x)\cdot\|t\|_\omega$ is $\mathcal B'_{X,\Omega'}$-measurable
and $E$ is countable.

Let $(\varphi_n)_{n\in\mathbb N} $ is a family of metric families on $L$. We assume that $(L,\varphi_n)$ forms a measurable line bundle on $X$, and that, for any $\omega\in\Omega'$, 
\[\lim_{n\rightarrow+\infty}d_\omega(\varphi_n,\varphi)=0.\]
If for any $n\in\mathbb N$, the measurable line bundle $(L,\varphi_n)$ verifies the assertion of the proposition, then so does $(L,\varphi)$. 
Therefore the proposition follows.
\end{proof}

\begin{coro}\label{Cor: measurability of green function}
Let $f \in \mathscr{C}^0_{\mathrm{m}}(X)$.
Let $L$ be an ample invertible $\mathcal{O}_X$-module and $\varphi$ be a family of semipositive metrics of $L$ such that
$(L, \varphi)$ is measurable.
If $\varphi + \mathrm{e}^{-f}$ is semipositive, then $f_{\Omega'}$ is  $\mathcal B'_{X,\Omega'}$-measurable.
\end{coro} 

\begin{proof}
We may assume that $L$ is very ample.
Let $s$ be a non-zero global section of $L$. Then, by Proposition~\ref{Pro: measurability of green function 02},
$|s|_{\varphi,\Omega}$ and $|s|_{\varphi + \mathrm{e}^{-f},\Omega}$ are $\mathcal B'_{X,\Omega'}$-measurable, and hence
$f_{\Omega'}$ is $\mathcal B'_{X,\Omega'}$-measurable because 
\[
f_{\Omega} = \ln \frac{|s|_{\varphi,\Omega}}{|s|_{\varphi + \mathrm{e}^{-f},\Omega}}.
\]
\end{proof}

\begin{theo}
We assume that $X^2+1$ is reducible in $K$.
Then $\mathcal B'_{X,\Omega_{\infty}} = \mathcal B_{X,\Omega_\infty}$. 
\end{theo}

\begin{proof}
Let $U$ be a non-empty Zariski open set of $X$ and $(f_{\omega})_{\omega \in \Omega_\infty}$
be a measurable family of continuous functions over $U$. 
It is sufficient to show that $f_{\Omega_\infty}$ is $\mathcal B'_{X,\Omega_{\infty}}$-measurable.
If we set $f_{+,\omega} = \max\{ f_{\omega}, 0\}$ and $f_{-,\omega} = \max\{ -f_{\omega}, 0 \}$,
then $f_{\omega} = f_{+, \omega} - f_{-,\omega}$, and $f_+ = \{ f_{+,\omega} \}_{\omega \in \Omega_\infty}$ and
$f_- = \{ f_{-,\omega} \}_{\omega \in \Omega_\infty}$ are measurable. Thus we may assume that $f \geqslant 0$.
Therefore the assertion follows from Theorem~\ref{theorem:pontwise:limit} and Corollary~\ref{Cor: measurability of green function}.
\end{proof}

\section{Equidistribution theorem}

Throughout this section, we assume that $S$ is proper.

\begin{defi}
Let $X$ be a reduced projective scheme over $\Spec K$ and $\overline L$ be a relatively nef adelic line bundle on $X$. For any integral closed subscheme $Y$ of $X$ such that $L|_Y$ is big, we define the \emph{normalized height} of $Y$ with respect to $\overline L$ as 
\[h_{\overline L}(Y):= \frac{(\overline L|_Y^{\dim(Y)+1})_S}{(\dim(Y)+1) (L|_Y^{\dim(Y)})}.\]
\end{defi}

\begin{theo}\label{theorem:equidistribution}
Let $X$ be an integral projective scheme over $\Spec K $ and $\overline L$ be an adelic line bundle on $X$. We assume that $L$ is big and semi-ample,
and that  $\varphi$ is semi-positive. Let $(Y_n)_{n\in\mathbb N}$ be a sequence of integral closed subschemes of $X$. Assume that,
\begin{enumerate}[label=\rm(\arabic*)]
\item the sequence $(Y_n)_{n\in\mathbb N}$ is generic, namely, for any strict closed subscheme $Z$ of $X$, the set $\{n\in\mathbb N\,:\,Y_n\subseteq Z\}$ is finite,
\item for any $n\in\mathbb N$, $L|_{Y_n}$ is big,
\item the sequence $(Y_n)_{n\in\mathbb N}$ is small, namely the sequence $(h_{\overline L}(Y_n))_{n\in\mathbb N}$ converges to $h_{\overline L}(X)$,
\end{enumerate}
Let $\Omega'$ be a measurable subset of $\Omega$ \textup{(}i.e. $\Omega' \in \mathcal A$\textup{)}. 
Then the sequence $(\delta_{\overline L,Y_n,\Omega'})_{n\in\mathbb N}$ converges weakly to $\delta_{\overline L,X,\Omega'}$. In other words, for any $f\in\mathscr C^0_{\mathrm{a}}(X)$, one has
\begin{equation}\label{Equ: equi distribution equality}\lim_{n\rightarrow+\infty}\delta_{\overline L,Y_n,\Omega'}(f)=\delta_{\overline L,X,\Omega'}(f).\end{equation}
\end{theo}
\begin{proof}
First we assume that $\Omega' = \Omega$.
For any $f\in\mathscr C^0_{\mathrm{a}}(X))$, let 
\[\Psi(f)=\liminf_{n\rightarrow+\infty}
\frac{\widehat{\mathrm{vol}}_{\chi}(\overline L(f)|_{Y_n})}{(\dim(Y_n)+1)\deg_L(Y_n))}.\]
By Corollary \ref{Cor: concavity of vol chi}, this is a concave function on $\mathscr C^0_{\mathrm{a}}(X)$. Since the sequence $(Y_n)_{n\in\mathbb N}$ is generic, $\Psi(f)$ is bounded from below by \[\widehat{\mu}_{\max}^{\mathrm{asy}}(\overline L(f))\geqslant\frac{\widehat{\operatorname{vol}}_{\chi}(\overline L(f))}{(d+1)\deg_L(X)}.\]
Moreover, the hypothesis of the theorem requires that the equality 
\begin{equation}\label{Equ: psi 0}\Psi(0)=\frac{\widehat{\operatorname{vol}}_{\chi}(\overline L(f))}{(d+1)\deg_L(X)}.\end{equation}
By Proposition \ref{Pro: differentiability of vol chi}, the function $f\mapsto \widehat{\operatorname{vol}}_{\chi}(\overline L(f))$ is G\^{a}teaux differentiable at $f=0$ and its differential is given by the linear form
\[f\longmapsto (d+1)\deg_L(X)\delta_{\overline L,X}(f).\]
Since $\Psi$ is a concave function, there exists a linear form $\ell: \mathscr C^0_{\mathrm{a}}(X)\rightarrow\mathbb R$ such that $\ell(f)+\Psi(0)\geqslant\Psi(f)$ for any $f\in\mathscr C^0_{\mathrm{a}(X)}$. We then deduce, by the equality \eqref{Equ: psi 0}, that $\ell(f)\geqslant \delta_{\overline L,X}(f)$ for any $f\in\mathscr C^0_{\mathrm{a}}(X)$, which leads to $\ell=\delta_{\overline L,X}(f)$. Thus $\delta_{\overline L,X}(\ndot)$ is the unique linear form on $\mathscr C^0_{\mathrm{a}}(X)$ such that $\delta_{\overline L,X}(f)+\Psi(0)\geqslant\Psi(f)$ for any $f\in\mathscr C^0_{\mathrm{a}}(X)$. 

By the condition (2) of the theorem, one has 
\[\Psi(0)=\lim_{n\rightarrow+\infty}
\frac{\widehat{\mathrm{vol}}_{\chi}(\overline L|_{Y_n})}{(\dim(Y_n)+1)\deg_L(Y_n))}\]
and hence
\[\Psi(f)-\Psi(0)=\liminf_{n\rightarrow+\infty}\frac{\widehat{\mathrm{vol}}_{\chi}(\overline L(f)|_{Y_n})-\widehat{\mathrm{vol}}_{\chi}(\overline L|_{Y_n})}{(\dim(Y_n)+1)\deg_L(Y_n)}.\]
Since the function 
\[(f\in\mathscr C^0_{\mathrm{a}}(X))\rightarrow \widehat{\mathrm{vol}}_{\chi}(\overline L(f)|_{Y_n})\]
is concave and G\^{a}teaux differentiable at $f=0$ with differential $\delta_{\overline L,Y_n}(\ndot)$, we obtain 
\[\Psi(f)-\Psi(0)\leqslant\liminf_{n\rightarrow+\infty}\delta_{\overline L,Y_n}(f).\]
Applying this inequality to $tf$ with $t>0$, and taking the limit when $t\rightarrow 0$, we obtain
\[\delta_{\overline L,X}(f)\leqslant\liminf_{n\rightarrow+\infty}\delta_{\overline L,Y_n}(f). \]
Replacing $f$ by $-f$, we obtain
\[\delta_{\overline L,X}(f)\geqslant\limsup_{n\rightarrow+\infty}\delta_{\overline L,Y_n}(f).\]
Therefore,
\[\delta_{\overline L,X}(f)=\lim_{n\rightarrow+\infty}\delta_{\overline L,Y_n}(f).\]

The general case is a consequence of the previous case and Corollary~\ref{coro:extension:by:zero}.
\end{proof}

\begin{rema} Let $L$ be a big invertible $\mathcal O_X$-module and
let $Z$ be its augmented base locus. Note that $Z$ is a proper closed subset of $X$. If $(Y_n)_{n\in\mathbb N}$ is a generic sequence of integral closed subschemes of $X$, the set $\{n\in\mathbb N\,:\,Y_n\subseteq Z\}$ is finite. Therefore, for sufficiently large $n\in\mathbb N$, the restriction of $L$ to $Y_n$ is big. This is a consequence of \cite[Theorem 1.5]{MR3673639} on the positivity of restricted volumes. 
\end{rema}

\begin{rema} Note that
\[
\begin{cases}
{\displaystyle \delta_{\overline L, Y_n, \Omega'}(f) = \int_{\Omega'} \nu(\mathrm{d} \omega) \int_{X_\omega^{\mathrm{an}}}f_\omega(x)\,\delta_{\overline{L}, Y_n, \omega} (\mathrm{d}x)}, \\[2ex]
{\displaystyle \delta_{\overline L, X, \Omega'}(f) = \int_{\Omega'} \nu(\mathrm{d} \omega) \int_{X_\omega^{\mathrm{an}}}f_\omega(x)\,\delta_{\overline{L}, X, \omega} (\mathrm{d}x)},
\end{cases}
\]
so if $\nu(\Omega') = 0$, then $\delta_{\overline L, Y_n, \Omega'}(f) = \delta_{\overline L, X, \Omega'}(f) = 0$. Thus Theorem~\ref{theorem:equidistribution} has a meaning in the case where $\nu(\Omega') > 0$.
\end{rema}

\begin{rema}
By Theorem~\ref{theorem:equidistribution}, 
 we can recover the equidistribution theorem over an arithmetic function field including a number field case (cf. \cite[Theorem~6.1]{MR1779799}), so the Bogomolov conjecture over an arithmetic function field is also derived.
\end{rema}

\section{Bogomolov's conjecture over a countable field of characteristic zero}

Throughout this section, we assume that $S$ is proper and that $K$ is algebraically closed, countable and of characteristic $0$. We assume in addition that $\nu(\Omega_\infty)>0$ and $\nu(\mathcal A) \not\subseteq \{ 0, 1 \}$.

Let $X$ be a projective integral scheme over $K$ and $\overline{L}$ be an adelic line bundle on $X$.
The \emph{essential minimum}  $\hat{\mu}_{\mathrm{ess}}(\overline{L})$ of $\overline{L}$ is defined to be
\[
\hat{\mu}_{\mathrm{ess}}(\overline{L}) := \sup_{Z \subsetneq X} \inf_{x \in (X \setminus Z)(K)} h_{\overline{L}}(x)
\]
where $Z$ runs over the set of all proper closed subsets of $X$.

Let $A$ be an abelian variety over $\Spec K$. 
For any integer $n$, let  $[n]:A\rightarrow A$ be the morphism of multiplication by $n$. Let $L$ be a symmetric ample invertible $\mathcal O_A$-module. 
If we fix an isomorphism $[2]^*(L) \simeq L^{\otimes 4}$, then, for each $\omega \in \Omega$,
we can assign the local canonical compactification $\varphi_{\omega}$ (with respect to $[2]$) to $L_{\omega}$.
If we set $\varphi = \{ \varphi_{\omega} \}_{\omega \in \Omega}$, then, by Proposition~\ref{prop:global:compactifcation},
$(L, \varphi)$ is a nef adelic line bundle on $X$.
Since $[2]$ and $[n]$ are commutative and $K$ is algebraically closed, by Remark~\ref{rem:compactibility:compactification},
we can find a suitable isomorphism $[n]^*(L) \simeq L^{\otimes n^2}$ such that
the local canonical compatification of $L_{\omega}$ with respect to $[n]$ coincides with $\varphi_\omega$.

\begin{theo}\label{thm:Bogomolov:conj}
Let $X$ be an integral subscheme of $A$ 
such that the stabilizer of $X$ is trivial. 
If $\dim X > 0$,
then the essential minimum 
$\hat{\mu}_{\mathrm{ess}}(\rest{\overline{L}}{X})$ of $X$ is strictly positive.
\end{theo}
\begin{proof}
There exists an integer $m\geqslant 1$ such that the morphism 
\[f:X^m\longrightarrow A^{m-1},\quad (x_1,\ldots,x_m)\longmapsto (x_2-x_1,\ldots,x_m-x_{m-1})\]
is birational onto its image but not finite. We assume by contradiction that there exists a generic sequence $(x_n)_{n\in\mathbb N}$ in $X(K)$ such that $h_{\overline L}(x_n)$ converges to $0$ when $n\rightarrow+\infty$. This sequence permits us to construct a generic sequence $(y_n)_{n\in\mathbb N}$ in $X^m$ such that $h_{\overline L^{\boxtimes m}}(y_n)$ converges to $0$ when $n\rightarrow+\infty$. 
Moreover, since N\'eron-Tate height is a quadratic form, we also deduce that $h_{\overline L^{\boxtimes(m-1)}}(f(y_n))$ converges to $0$ when $n\rightarrow+\infty$,
that is, \[\lim_{n\to\infty} h_{f^*(\overline L^{\boxtimes(m-1)})}(y_n) = 0.\]

Applying the equidistribution theorem  (Theorem \ref{theorem:equidistribution}) to the sequence $(y_n)_{n\in\mathbb N}$, 
we deduce that the sequences of measures
\[\delta_{\overline L^{\boxtimes m},y_n,\Omega_\infty}\quad\text{and}\quad\delta_{f^*(\overline L^{\boxtimes(m-1)}),y_n,\Omega_\infty},\quad n\in\mathbb N\]
converge weakly to $\delta_{\overline L^{\boxtimes m},X^m,\Omega_\infty}$ and $\delta_{f^*(\overline L^{\boxtimes(m-1)}),X^m,\Omega_\infty}$, respectively, 
and hence
\[\delta_{\overline L^{\boxtimes m},X^m,\Omega_\infty}=\delta_{f^*(\overline L^{\boxtimes (m-1)}), X^m,\Omega_\infty}.\]
Let $X_{\operatorname{reg}}$ be the regular locus of $X$, viewed as an open subscheme of $X$. 
Then, for any $\omega\in\Omega$, there exists a unique smooth and non-negative function $q_\omega(\ndot)$ on $(X^m_{\operatorname{reg}})_{\omega}^{\mathrm{an}}$ 
such that, if we extend the domain of definition of $q_{\omega}$ to $(X^{m})_\omega^{\mathrm{an}}$ by taking the value $0$ on $(X^{m})_\omega^{\mathrm{an}}\setminus (X_{\operatorname{reg}}^m)_\omega^{\mathrm{an}}$ 
(which is of measure $0$ with respect to $\delta_{\overline L^{\boxtimes m},X^m,\omega}$), the following relation holds
\[q\,\delta_{\overline L^{\boxtimes m},X^m,\Omega_\infty}=\delta_{f^*(\overline L^{\boxtimes m-1}),X^m,\Omega_\infty},\]
where $q(x) = q_\omega(x)$ for $x \in X_\omega^{\mathrm{an}}$.

Let $\pi : X^m \to \PP_K^{m\dim X}$ be a projection and $V$ be an affine Zariski open set of $\PP^{m\dim X}_K$ such that
if we set $U = \pi^{-1}(V)$, then $U \subseteq X_{\operatorname{reg}}^m$ and $\pi : U \to V$ is \'{e}tale.
By Proposition~\ref{prop:measurable:partial:derivative}, 
for any rational point $x$ of $U$, the function
\[(\omega\in\Omega_\infty)\longmapsto q_\omega(x)\]
is $\mathcal A$-measurable. In particular, if we consider  
a function $p_\omega$ on $X_\omega^{\mathrm{an}}$ given by
\[
p_{\omega}(x) = \begin{cases}
q_\omega(x) & \text{if $x \in U_\omega^{\mathrm{an}}$}, \\
0 & \text{otherwise},
\end{cases}
\]
then $(p_{\omega})_{\omega \in \Omega_{\infty}}$ is actually a measurable function 
family, so that the function
\[(X^{m})^{\mathrm{an}}_{\Omega_\infty}\longrightarrow\mathbb R,\quad (x\in (X^{m})^{\mathrm{an}}_{\Omega_\infty})\longmapsto p_\omega(x)\]
is $\mathcal B_{X,\Omega_\infty}$-measurable. 
Moreover, the equation  \[p\,\delta_{\overline L^{\boxtimes m},X^m,\Omega_\infty}=\delta_{f^*(\overline L^{\boxtimes m-1}),X^m,\Omega_\infty}\] holds as measures (see Remark \ref{Rem: unique determined by integrals along C}), where $p(x)=p_\omega(x)$ when $x\in X_\omega^{\mathrm{an}}$.
Therefore, by Proposition~\ref{prop:Radon-Nikodym},
there exists $\Omega'\subseteq\Omega_\infty$ such that $\nu(\Omega_\infty\setminus\Omega')=0$ and that the equality $p_\omega(x)=1$
holds $\delta_{\overline L^{\boxtimes m},X^m,\omega}$-almost everywhere for any $\omega\in
\Omega'$, which implies \[q_\omega(x) = 1\] $\delta_{\overline L^{\boxtimes m},X^m,\omega}$-almost everywhere on $X_{\omega}^{m,\mathrm{an}}$ for any $\omega\in
\Omega'$. This leads to a contradiction since $q_\omega$ is continuous on $X_{\operatorname{reg},\omega}^{m,\mathrm{an}}$ and vanishes at the diagonal points of it.
\end{proof}

As a consequence of the above theorem, we have the following answer of Bogomolov's conjecture for $K$.

\begin{coro}
Let $X$ be an integral subscheme of $A$. If $\hat{\mu}_{\mathrm{ess}}(\rest{\overline{L}}{X}) = 0$, 
then  $X$ is a translation of an abelian subvariety by a closed point of height $0$.
\end{coro}

\begin{proof}
In the same argument as \cite[the last paragraph in the proof of Theorem~9.20]{MArakelov}, we may assume that the stabilizer of $X$ is trivial.
Thus, by Theorem~\ref{thm:Bogomolov:conj}, one has $\dim X = 0$, so we set $X = \{ x \}$.
Thus $\hat{\mu}_{\mathrm{ess}}(\rest{\overline{L}}{X}) = h_{\overline{L}}(x) = 0$, as required.
\end{proof}

\begin{rema}
Assume that any finitely generated subfield of $K$ has Northcott's property.
Then any closed point of height 0 with respect to the N\'{e}ron-Tate height on the abelian variety $A$ is a torsion point.
Indeed, we choose a subfield $K'$ of $K$ such that
$A$, $L$ and $x$ are defined over $K'$ and $K'$ is finitely generated over $\mathbb Q$. Then, by Northcott's property,
$\{ n x \mid n \in \mathbb Z \}$ is a finite group because $h_{\overline{L}}(nx) = 0$ for all $n \in \mathbb Z$,  so $x$ is a torsion point.

The geometric analogue of Bogomolov's conjecture for Abelian varieties over function fields has been proved by a series of works of Gubler \cite{MR2318560}, Yamaki \cite{MR3117164,MR3505646,MR3747175,MR3671938,MR3836145}, Gao-Habbegger \cite{MR3922127}, Cantat-Gao-Habegger-Xie \cite{MR4202494} and Xie-Yuan \cite{MR4448992}. It is an interesting question to investigate the condition (on the polarized Abelian variety $(A,L)$) under which the result of Theorem \ref{thm:Bogomolov:conj} holds without the assumption $\nu(\Omega_\infty)>0$. 
\end{rema}

\appendix 

\section{Canonical compactification}

In this appendix, we recall several basic facts on the canonical compactification.

\subsection{Local canonical compactification}

Let $(k,|\ndot|)$ be a field equipped with a complete absolute value.
Let $X$ be a geometrically integral projective variety over $k$.
In particular, $H^0(X, \mathcal{O}_X) = k$.
Let $f : X \to X$ be a surjective endomorphism of $X$ over $k$ and $L$ be a semiample line bundle on $X$.
We assume that there exists an isomorphism $\alpha : f^*(L) \simeq L^{\otimes d}$ for some integer $d \geqslant 2$. 
It is well-known that there exists a unique semipositive metric $\varphi_{f,\alpha}$ of $L^{\mathrm{an}}$
such that $\alpha$ induces an isometry
$f^*(L, \varphi_{f,\alpha}) \simeq (L, \varphi_{f,\alpha})^{\otimes d}$ (cf. \cite[Proposition~2.5.11]{CMArakelovAdelic}), that is,
$|\ndot|_{f^*(\varphi_{f,\alpha})} = |\alpha(\ndot)|_{d \varphi_{f,\alpha}}$. The metric  $\varphi_{f,\alpha}$ is called the (local) canonical compactification of $L$. 
Let us begin with the following lemma.

\begin{lemm}\label{lem:dynamic:const}
Let $\lambda$ be a continuous function on $X^{\mathrm{an}}$, $a \in \mathbb{R}$ and $b \in \mathbb{R}_{>1}$.
If $f^*(\lambda) = b \lambda + a$, then $\lambda$ is a constant and $\lambda = -a/(b-1)$.
\end{lemm}

\begin{proof}
By our assumption,
\[
 \max\limits_{x \in X^{\mathrm{an}}} \{ \lambda(x) \} = b \max\limits_{x \in X^{\mathrm{an}}} \{ \lambda(x) \} + a \quad\text{and}\quad
\min\limits_{x \in X^{\mathrm{an}}} \{ \lambda(x) \} = b \min\limits_{x \in X^{\mathrm{an}}} \{ \lambda(x) \} + a,
\]
that is, \[(b-1) \max\limits_{x \in X^{\mathrm{an}}} \{ \lambda(x) \} = (b-1)\min \limits_{x \in X^{\mathrm{an}}} \{ \lambda(x) \} = -a,\] and hence
the assertion follows.
\end{proof}

\begin{prop}\label{prop:change:isom}
If we change the isomorphism $\alpha$ by $c\alpha$ $(c \in k^{\times})$,
then \[|\ndot|_{\varphi_{f,c\alpha}} =|c|^{-1/(d-1)} |\ndot|_{\varphi_{f,\alpha}}.\]
\end{prop}

\begin{proof}
Indeed, we can find a continuous function $\lambda$ on $X^{\mathrm{an}}$ such that
$|\ndot|_{\varphi_{f,c\alpha}} = \exp(\lambda)|\ndot|_{\varphi_{f,\alpha}}$.
Thus, 
\begin{align*}
 \exp(f^*(\lambda))|\ndot|_{f^*(\varphi_{f,\alpha})}  & =  |\ndot|_{f^*(\varphi_{f,c\alpha})} = |(c\alpha)(\ndot)|_{d \varphi_{f,c\alpha}} = 
|c| \exp(d\lambda) |\alpha(\ndot)|_{d\varphi_{f,\alpha}} \\
& = \exp(d\lambda + \log |c|)|\ndot|_{f^*(\varphi_{f,\alpha})},
\end{align*}
and hence $f^*(\lambda) = d \lambda + \log |c|$. Therefore, by Lemma~\ref{lem:dynamic:const}, $\lambda$ is constant and
$\lambda = -\log |c|/(d-1)$, as required.
\end{proof}

Let $g : X \to X$ be another surjective endomorphism of $X$ such that 
there exists an isomorphism $\beta : g^*(L) \simeq L^{\otimes e}$ for some integer $e \geqslant 2$.
We assume that $f \circ g = g \circ f$. Let us consider the following homomorphisms
\[
\begin{cases}
\begin{CD}g^*(f^*(L)) @>\sim>{g^*(\alpha)}> g^*(L^{\otimes d}) @>\sim>{\beta^{\otimes d}}> L^{\otimes de}\end{CD}, \\[3ex]
\begin{CD} g^*(f^*(L)) = f^*(g^*(L)) @>\sim>{f^*(\beta)}> f^*(L^{\otimes e}) @>\sim> {\alpha^{\otimes e}}>  L^{\otimes de} \end{CD}.
\end{cases}
\]
Then there exists $r \in k^{\times}$ such that
$\beta^{\otimes d} \circ g^*(\alpha) = r \cdot \alpha^{\otimes e} \circ f^*(\beta)$.
In the case where $r=1$, we say that $(f, \alpha)$ is compatible with $(g, \beta)$.

\begin{prop}\label{prop:change:endo}
One has $|\ndot|_{\varphi_{g,\beta}} = |r|^{-1/(d-1)(e-1)}|\ndot|_{\varphi_{f,\alpha}}$.
In particular, if $(f, \alpha)$ is compatible with $(g, \beta)$, then $\varphi_{f,\alpha} = \varphi_{g,\beta}$.
\end{prop}

\begin{proof}
We can find a continuous function $\lambda$ on $X^{\mathrm{an}}$ such that 
\[|\ndot|_{g^*(\varphi_{f,\alpha})} = \exp(\lambda)|\beta(\ndot)|_{ e\varphi_{f,\alpha}}.\]
Thus
\begin{align*}
|\ndot|_{f^*(g^*(\varphi_{f,\alpha}))} & = \exp(f^*(\lambda))|f^*(\beta)(\ndot)|_{ ef^*(\varphi_{f,\alpha})} = \exp(f^*(\lambda))|\alpha^{\otimes e}(f^*(\beta)(\ndot))|_{de\varphi_{f,\alpha}}.
\end{align*}
On the other hand,
\begin{align*}
|\ndot|_{g^*(f^*(\varphi_{f,\alpha}))} & = |g^*(\alpha)(\ndot)|_{g^*(d\varphi_{f,\alpha})} = \exp(\lambda)^d |\beta^{\otimes d}(g^*(\alpha)(\ndot))|_{de\varphi_{f,\alpha}} \\
& = \exp(\lambda)^d |r\alpha^{\otimes e}(f^*(\beta)(\ndot))|_{de\varphi_{f,\alpha}} =\exp(\lambda)^d |r||\alpha^{\otimes e}(f^*(\beta)(\ndot))|_{de\varphi_{f,\alpha}},
\end{align*}
so 
$f^*(\lambda) = d \lambda + \log |r|$.
Therefore, by Lemma~\ref{lem:dynamic:const},
$\lambda$ is a constant function and $\lambda = -\log|r|/(d-1)$.

Here we set $|\ndot|_{\varphi_{g,\beta}} = \exp(\mu)|\ndot|_{\varphi_{f,\alpha}}$  for some continuous function $\mu$ on $X^{\mathrm{an}}$.
Then, as $|\ndot|_{g^*(\varphi_{g,\beta})} = |\beta(\ndot)|_{e\varphi_{g,\beta}}$, one has
\begin{align*}
\exp(g^*(\mu)) |\ndot|_{g^*(\varphi_{f,\alpha})} & = |\ndot|_{g^*(\varphi_{g,\beta})}= |\beta(\ndot)|_{e\varphi_{g,\beta}} =
\exp(e\mu) |\beta(\ndot)|_{e \varphi_{f,\alpha}} \\
& = \exp(e\mu-\lambda) |\ndot|_{g^*(\varphi_{f,\alpha})},
\end{align*}
that is, $g^*(\mu) =  e\mu - \lambda$. Therefore, by Lemma~\ref{lem:dynamic:const}, $\mu$ is a constant and
$\mu = \lambda/(e-1)$, and hence
\[
\mu = -\log|r|/(d-1)(e-1),
\]
as required.
\end{proof}

\begin{rema}\label{rem:compactibility:compactification}
If  $\beta^{\otimes d} \circ g^*(\alpha) = r  \cdot \alpha^{\otimes e} \circ f^*(\beta)$ ($r \in k^{\times}$), then, for $c \in k^{\times}$,
\[
(c\beta)^{\otimes d} \circ g^*(\alpha) = c^d \cdot \beta^{\otimes d} \circ g^*(\alpha) = (c^d  r ) \cdot \alpha^{\otimes e} \circ f^*(\beta) =
(c^{d-1} r) \cdot \alpha^{\otimes e} \circ f^*(c\beta).
\]
Thus if there exists $c \in k$ such that $c^{d-1} r = 1$, then
$(f, \alpha)$ is compatible with $(g, c\beta)$.
In particular, if $k$ is algebraically closed, then, for $(f, \alpha)$,
we can find $\beta$ such that $(f, \alpha)$ is compatible with $(g, \beta)$.
\end{rema}

\subsection{Global canonical compactification}
Let $K$ be a field equipped with an adelic structure $((\Omega, \mathcal{A}, \nu), \phi)$.
We assume that $\nu(\mathcal{A}) \not\subseteq \{ 0, \infty \}$.
Let $X$ be a geometrically integral projective variety over $K$ and $f : X \to X$ be a surjective endomorphism of $X$ over $K$.
Let $L$ be a semiample line bundle on $X$ such that there exists an isomorphism $\alpha : f^*(L) \overset{\sim}{\longrightarrow} L^{\otimes d}$ for some integer $d \geqslant 2$.
For each $\omega \in \Omega$, let $\varphi_\omega$ be the local compactification of $L_\omega$, that is,
the isomorphism $\alpha_\omega : f_\omega^*(L_\omega)  \overset{\sim}{\longrightarrow}  L_\omega^{\otimes d}$  induces an isometry
$f_{\omega}^*(L_\omega, \varphi_\omega) \simeq (L_\omega, \varphi_\omega)^{\otimes d}$.

\begin{prop}\label{prop:global:compactifcation}
If we set $\varphi = \{ \varphi_{\omega} \}_{\omega \in \Omega}$, then $(L, \varphi)$ is a nef adelic line bundle on $X$.
The family $\varphi$ is called the global compactification of $L$.
\end{prop}

\begin{proof}
Let $\varphi_0 = \{ \varphi_{0, \omega} \}_{\omega \in \Omega}$ be a metric family of $L$ such that $(L, \varphi_0)$ is a relatively nef adelic line bundle.
By the assumption ``$\nu(\mathcal{A}) \not\subseteq \{ 0, \infty \}$'', we can find a non-negative measurable function $\vartheta$ on $\Omega$ such that
\[
\hat{\mu}_{\min}^{\mathrm{asy}}(L, \varphi_0) + \int_{\Omega} \vartheta(\omega) \nu(\mathrm{d}\omega) \geqslant 0.
\]
If we set $\varphi^{\vartheta}_{0, \omega} = \exp(-\vartheta(\omega)) \varphi_{0, \omega}$ and $\varphi_0^{\vartheta} = \{ \varphi_{0,\omega}^{\vartheta}\}_{\omega \in \Omega}$, then
it is easy to see that
\[
\hat{\mu}_{\min}^{\mathrm{asy}}(L, \varphi^{\vartheta}_0) = \hat{\mu}_{\min}^{\mathrm{asy}}(L, \varphi_0) +  \int_{\Omega} \vartheta(\omega) \nu(\mathrm{d}\omega) \geqslant 0,
\]
so $(L, \varphi^{\vartheta}_0)$ is nef by \cite[Proposition~5.3.6]{CMHS}. Thus we may assume that $(L, \varphi_0)$ is nef.

For each $\omega \in \Omega$,
one can find a continuous function $\lambda_\omega$ on $X_\omega^{\mathrm{an}}$ such that
\[|\ndot|_{f^*(\varphi_{\varphi_{0,\omega}})} = |\alpha_\omega(\ndot)|_{d\varphi_{0,\omega}} \exp(\lambda_\omega).\]
Note that $(\OO_X, \{\exp( \lambda_{\omega})|\ndot|_\omega \}_{\omega \in \Omega})$ is an adelic line bundle. 
Let
\[
h_{0,\omega} = 0\quad\text{and}\quad h_{n,\omega} = \sum_{i=0}^{n-1} \frac{1}{d^{i+1}} (f_\omega^i)^*(\lambda_{\omega})\quad (n \geqslant 1),
\]
and $|\ndot|_{\varphi_{n, \omega}} = |\ndot|_{\varphi_{0, \omega}} \exp(h_{n,\omega})$.
Then, in the same way as \cite[Proposition~2.5.11]{CMArakelovAdelic}, one can see that
$\{ h_{n,\omega} \}_{n=0}^{\infty}$ converges uniformly to a continuous function $h_{\infty, \omega}$ on $X_{\omega}^{\mathrm{an}}$
and $\alpha_\omega : f_\omega^*(L_\omega)  \overset{\sim}{\longrightarrow}  L_\omega^{\otimes d}$ induces an isometry 
\[f_{\omega}^*(L_{\omega}, \{ \varphi_{n-1, \omega}\}_{\omega \in \Omega}) \simeq (L_{\omega}, \{\varphi_{n,\omega}\}_{\omega \in \Omega} )^{\otimes d}.\]
In particular, if we set $|\ndot|_{\varphi_{\infty, \omega}} = |\ndot|_{\varphi_{0,\omega}}\exp(h_{\infty, \omega})$, then $\alpha_{\omega}$ yields
an isometry 
\[f_{\omega}^*(L_{\omega}, \{ \varphi_{\infty, \omega}\}_{\omega \in \Omega}) \simeq (L_{\omega}, \{\varphi_{\infty,\omega}\}_{\omega \in \Omega} )^{\otimes d},\]
and hence $ \varphi_{\infty, \omega}$ is the local canonical compactification of $L_\omega$.
Thus, by the uniqueness of the local canonical compactfication, we have $\varphi_{\infty,\omega} = \varphi_{\omega}$ for all $\omega \in \Omega$.
Let $\varphi_{n} = \{ \varphi_{n, \omega} \}_{\omega \in \Omega}$. 
By \cite[Proposition~6.1.29]{CMArakelovAdelic}, $(L, \varphi)$ is measurable because
$(L, \varphi_{n})$ is measurable for all $n$. Moreover, in the same way as \cite[Proposition~2.5.11]{CMArakelovAdelic}, we obtains
\[
d_{\omega}(\varphi, \varphi_{0}) \leqslant \frac{\|\lambda_{\omega}\|_{\sup}}{d-1},
\]
so $(\omega \in \Omega) \mapsto d_{\omega}(\varphi, \varphi_{g_0})$ is dominated.
Thus $(L, \varphi)$ is dominated by \cite[Proposition~6.1.12]{CMArakelovAdelic}.
Further, as $f^*(L, \varphi_{n-1}) \simeq (L, \varphi_{n})^{\otimes d}$, we can see that
$(L, \varphi_{n})$ is nef for all $n$. Therefore, $(L, \varphi)$ is also nef, as required.
\end{proof}

\begin{rema}
We assume that the adelic structure is proper

(1) Let $\alpha : f^*(L) \to L^{\otimes d}$ be the isomorphism.  
If we change the isomorphism $\alpha$ by $c\alpha$ ($c \in K^{\times}$), then, by Proposition~\ref{prop:change:isom},
\[|\ndot|_{\varphi_{f, c\alpha, \omega}} = |c|_{\omega}^{-1/(d-1)}|\ndot|_{\varphi_{f, \alpha, \omega}}\] for all $\omega \in \Omega$.
Thus, by the product formula, one has
$h_{(L, \varphi_{f, c\alpha})} = h_{(L, \varphi_{f, \alpha})}$.

(2) Let $g : X \to X$ be another surjective endomorphism of $X$ over $K$.
We assume that $f \circ g = g \circ f$ and there exists an isomorphism $\beta : g^*(L) \to L^{\otimes e}$
for some integer $e \geqslant 2$. Then, by Proposition~\ref{prop:change:endo},
there exists $r \in K^{\times}$ such that
\[
|\ndot|_{\varphi_{f, \alpha,\omega}} = |r|_{\omega}^{-1/(d-1)(e-1)} |\ndot|_{\varphi_{g, \beta,\omega}}.
\]
for all $\omega \in \Omega$. Therefore, by the product formula,
one has $h_{(L, \varphi_{f, \alpha})} = h_{(L, \varphi_{g, \beta})}$.
\end{rema}

\bibliography{positivity}
\bibliographystyle{plain}

\printindex

\end{document}